\documentclass[final,twoside,11pt]{entics} 
\usepackage{enticsmacro}
\usepackage{graphicx, amsmath}
\usepackage[all]{xy}
\newcommand\eqdef{\stackrel{\textrm{{\scriptsize def}}}{=}}
\newcommand\diff{\setminus}

\newcommand\nat{\mathbb{N}}
\newcommand\rat{\mathbb{Q}}

\newcommand{\real}{\mathbb{R}}
\newcommand\Ipref{\mathbf{I}}
\newcommand\IR{\Ipref\real}
\newcommand\realp{\real_+}
\newcommand{\creal}{\overline\real_+}

\newcommand\IRbb{\IR^\star} 
\newcommand\Icreal{\IR_+^\star} 
\newcommand\identity[1]{\mathrm{id}_{#1}}

\newcommand\Lform{\mathcal{L}}
\newcommand\Val{\mathbf{V}}
\newcommand\lebesgue{\overline{\lambda}}

\newcommand\upc{\mathop{\uparrow}}
\newcommand\dc{\mathop{\downarrow}}
\newcommand\uuarrow{\rlap{$\uparrow$}\raise.5ex\hbox{$\uparrow$}}
\newcommand\ddarrow{\rlap{$\downarrow$}\raise.5ex\hbox{$\downarrow$}}
\newcommand\Open{{\mathcal O}}
\newcommand\fin{\mathrm{f}}
\newcommand\supp{\mathrm{supp}\;}

\newcommand\cb[1]{\mathbf{#1}} 
\newcommand{\catc}{{\cb{C}}} 

\newcommand{\Dcpo}{{\mathbf{Dcpo}}}

\newcommand\rval[1]{\widetilde{#1}}
\newcommand\Rp{\real_+}

\newcommand\directed{\sideset{}{^{\,\makebox[0pt]{$\scriptstyle\uparrow$}\;}}}
\newcommand\filtered{\sideset{}{^{\,\makebox[0pt]{$\scriptstyle\downarrow$}\;}}}
\newcommand\dsup{\directed\sup}
\newcommand\finf{\filtered\inf}

\newcommand\fcap{\filtered\bigcap}

\newcommand\uuuarrow{\rlap{$\uparrow$}\raise.5ex\hbox{$\uuarrow$}}

\newcommand\One[1]{\mathbf{1}_{#1}}

\newcommand\patch{\mathsf{patch}}

\def\conf{ISDT 2022} 	
\volume{2}			
\def\volu{2}			
\def\papno{2}			
\def\mydoi{{\normalshape  \href{https://doi.org/10.46298/entics.10351}{10.46298/entics.10351}}}
\def\copynames{J. Goubault-Larrecq, X. Jia} 
\def\runtitle{Continuous $R$-valuations}

\def\lastname{Goubault-Larrecq, and Jia}
\begin{document}
\begin{frontmatter}
  \title{Continuous $R$-valuations} 
  \author{Jean Goubault-Larrecq\thanksref{a}\thanksref{myemail}}	
   \author{Xiaodong Jia\thanksref{b}\thanksref{ALL}\thanksref{coemail}}		
   \address[a]{Universit\'e Paris-Saclay, CNRS, ENS Paris-Saclay, Laboratoire M\'ethodes Formelles, 91190, Gif-sur-Yvette, France}  							
  \thanks[ALL]{Xiaodong Jia acknowledges the support of NSFC (No.\ 12001181, No.\ 12231007).}   
   \thanks[myemail]{Email: \href{mailto: goubault@lsv.fr} {\texttt{\normalshape
     goubault@lsv.fr}}} 
  \address[b]{School of Mathematics, Hunan University,
  Changsha, Hunan 401182, China} 
  \thanks[coemail]{Email:  \href{mailto: jia.xiaodong@yahoo.com} {\texttt{\normalshape
	 jia.xiaodong@yahoo.com}}}
\begin{abstract} 
  We introduce continuous $R$-valuations on directed-complete posets
  (dcpos, for short), as a generalization of continuous valuations in
  domain theory, by extending values of continuous valuations from
  reals to so-called Abelian d-rags $R$.
  
  Like the valuation monad $\Val$ introduced by Jones and Plotkin, we
  show that the construction of continuous $R$-valuations extends to a
  strong monad $\Val^R$ on the category of dcpos and Scott-continuous
  maps.  Additionally, and as in recent work by the two authors and
  C. Th\'eron, and by the second author, B. Lindenhovius, M. Mislove and
  V. Zamdzhiev, we show that we can extract a commutative monad
  $\Val^R_m$ out of it, whose elements we call minimal $R$-valuations.

  We also show that continuous $R$-valuations have close connections
  to measures when $R$ is taken to be $\Icreal$, the interval domain of
  the extended nonnegative reals: (1) On every coherent topological space, 
  every non-zero, bounded $\tau$-smooth measure $\mu$ (defined on the Borel 
  $\sigma$-algebra), canonically determines a continuous $\Icreal$-valuation;
  and (2) such a continuous $\Icreal$-valuation is the most precise (in a certain
  sense)  continuous $\Icreal$-valuation that approximates $\mu$, when the
  support of $\mu$ is a compact Hausdorff subspace of a second-countable
  stably compact topological space. This in particular applies to Lebesgue
  measure on the unit interval. As a result, the Lebesgue measure can be identified 
  as a continuous $\Icreal$-valuation.  Additionally, we show that the
  latter is minimal.

\end{abstract}
\end{frontmatter}

\section{Introduction}
The probability of an event is most often than not understood as a
real number between $0$ and $1$, and measures, as well as continuous
valuations, take their values in $\creal$, the set of non-negative
real numbers extended with $+\infty$.  What is there that is so
special with real numbers, and can we replace $\creal$ by some
elements in some other structure?  The question was once asked by
Vincent Danos to the first author, and came back to the authors in an
attempt to formulate an alternative to measures and continuous
valuations with values taken as \emph{exact reals}, in the sense of
Real PCF \cite{escardo96a,escardo96,edalat97,Plume:ERC} for example.
Exact reals are modeled there as \emph{intervals} that enclose the
true value that is intended, and computation proceeds by refining
these intervals further and further.  Indeed, one of
  the points of this paper is that we can extend continuous valuations
  to an interval-valued form of continuous valuations, with an
  interval-valued integration theory.  In addition, this also leads us to commutative
  valuations monads with intervals as values on the category of dcpos
  and Scott-continuous maps.

We should warn the reader that such an endeavor is probably useless
for computation purposes.  In the setting of type 2 theory of
effectivity, Weihrauch has shown that, under reasonable assumptions,
the map that sends a representable measure $\mu$ on $[0, 1]$ to
$\mu [0, 1/2)$ cannot be continuous if the target space $[0, 1]$ is
given the Scott topology of the reverse ordering $\geq$
\cite[Theorem~2.7]{Weihrauch:prob}.  (It is continuous with respect to
the usual ordering $\leq$.  In passing, type 2 theory of effectivity
on the reals does not differ much from ordinary domain-based notions
of computability, as Schulz has shown \cite{Schulz:TTE=RealPCF}.)
This roughly means that if we insist on representing $\mu [0, 1/2)$ as
\emph{precise} intervals, namely as intervals $[a, a]$ with the same
left and right bounds, then the rightmost $a$ will evolve in a
discontinuous manner.

Despite this, we will show in Section~\ref{sec:meas-as-cont} that
quite a lot of measures (in the ordinary sense) have representations
as interval-valued ``measures'' with \emph{precise} intervals; see
Remark~\ref{rem:rval}, in particular.  Before that, we will have to
define what we mean by ``measure'' with values in a domain of
intervals.  We will start (after recapitulating some preliminary
notions and results in Section~\ref{sec:preliminaries}) by giving a
pretty general possible answer to V. Danos' question in
Section~\ref{sec:rugs-d-rugs}: we may safely replace $\creal$ by any
structure of a kind that we call an \emph{Abelian d-rag}, which is a
weaker form of Abelian semiring, with a compatible ordering that turns
it into a dcpo.  This will allow us to define a notion of
\emph{continuous $R$-valuation} on a space $X$, for any Abelian d-rag
$R$, in Section~\ref{sec:cont-r-valu}.  The obvious definition would
be as a function from the open subsets of $X$ satisfying certain
requirements, but those requirements have proved elusive, especially
in Abelian d-rags where the additive zero $0$ differs from the bottom
element $\bot$, as with the domain of intervals.  For example, would
you define the measure of the empty set as $0$ or as $\bot$?  One is
needed for algebraic reasoning, so that adding the measure of the
empty set does nothing; the other is needed for approximation
purposes, e.g., in order to define the integral of a function $f$ as
the supremum of simpler sums.  We sidestep the issue by defining our
continuous R-valuations as being directly the functionals that one
would usually obtain by defining an integral out of a measure.  In
Section~\ref{sec:monads-continuous-r}, we show that continuous
R-valuations, much like continuous valuations \cite{jones89,jones90},
can be organized to form a strong monad on the category $\Dcpo$ of
dcpos and Scott-continuous maps.  Furthermore, as in
\cite{goubault-jia-theron-21} and~\cite{jia2021commutative}, one can
carve out a \emph{commutative} monad of so-called \emph{minimal}
$R$-valuations from the latter.  We start to examine the relationship
between measures and continuous $R$-valuations when $R$ is either
$\creal$ or the interval domain $\Icreal$ in
Section~\ref{sec:cont-r-valu-1}.  That section is devoted to a few
simple facts, and notably to the fact that every continuous
$\Icreal$-valuation induces an ordinary continuous
($\creal$-)valuation, which we call it view from the left.  In
Section~\ref{sec:meas-as-cont}, we will see that every non-zero,
bounded $\tau$-smooth measure $\mu$ on a coherent topological space
gives rise to a continuous $\Icreal$-valuation $\rval\mu$ in a natural
way, and that $\rval\mu$ is \emph{precise} in the sense alluded to
above.  In Section~\ref{sec:repr-meas-iii}, under slightly different
assumptions, we study the continuous $\Icreal$-valuations that
approximate a given, not necessarily bounded, measure, and we show
that there is a \emph{most precise} one; it so happens that this is
$\rval\mu$, once again.  In all those cases, there is no reason why
$\rval\mu$ should be minimal.  In
Section~\ref{sec:lebesgue-r-valuation}, we illustrate the question
with the Lebesgue measure $\lambda$ on $[0, 1]$ and its associated
continuous $\Icreal$-valuation $\rval\lambda$.  We note that
$\rval\lambda$ is not minimal.  However, we will show that replacing
$\lambda$ by its image measure under the inclusion of $[0, 1]$ into a
dcpo of intervals $\IRbb$ \emph{does} yield a minimal
$\Icreal$-valuation.  We conclude in Section~\ref{sec:conc}.

\section{Preliminaries}
\label{sec:preliminaries}

We refer to \cite{billingsley86} for basics of measure theory, and to
\cite{abramsky94,gierz03,goubault13a} for basics of domain theory and
topology.

\paragraph{Measure theory.}

A \emph{$\sigma$-algebra} on a set $X$ is a collection of subsets
closed under countable unions and complements.  A \emph{measurable
  space} $X$ is a set with a $\sigma$-algebra $\Sigma_X$.  The
elements of $\Sigma_X$ are usually called the \emph{measurable
  subsets} of $X$.

A \emph{measure} $\mu$ on $X$ is a $\sigma$-additive map from
$\Sigma_X$ to $\creal \eqdef \realp \cup \{+\infty\}$, where $\creal$ is the set of extended
non-negative real numbers .  The property of
\emph{$\sigma$-additivity} means that, for every countable family of
pairwise disjoint sets $E_n$,
$\mu (\bigcup_n E_n) = \sum_n \mu (E_n)$.  (Here $n$ ranges over any
subset of $\nat$, possibly empty.)

A \emph{measurable map} $f \colon X \to Y$ between measurable spaces
is a map such that $f^{-1} (E) \in \Sigma_X$ for every
$E \in \Sigma_Y$.  The \emph{image measure} $f [\mu]$ of a measure
$\mu$ on $X$ is defined by $f [\mu] (E) \eqdef \mu (f^{-1} (E))$.

The $\sigma$-algebra $\Sigma (A)$ \emph{generated by} a family $A$ of
subsets of $X$ is the the smallest $\sigma$-algebra containing $A$.
The \emph{Borel $\sigma$-algebra} on a topological space is the
$\sigma$-algebra generated by its topology.  The \emph{standard
  topology} on $\creal$ is generated by the intervals $[0, b[$,
$]a, b[$ and $]a, +\infty$, with $0 < a < b < +\infty$.  Its Borel
$\sigma$-algebra is also generated by the intervals $]a, +\infty]$
along (the Scott-open subsets, see below).  Hence a measurable map
$h \colon X \to \creal$ is a map such that
$h^{-1} (]t, +\infty]) \in \Sigma$ for every $t \in \real$.  Its
\emph{Lebesgue integral} can be defined elegantly through
\emph{Choquet's formula}:
$\int_X h d\mu \eqdef \int_0^{+\infty} \mu (h^{-1} (]t, +\infty]))
dt$, where the right-hand integral is an ordinary Riemann integral.

This formula makes the following \emph{change-of-variables formula} an
easy observation: for every measurable map $f \colon X \to Y$, for
every measurable map $h \colon Y \to \creal$,
$\int_Y h df[\mu] = \int_X (h \circ f) d\mu$.

The \emph{monotone convergence theorem} states that, given any measure
$\mu$ on a measurable space $X$, given any sequence
${(h_n)}_{n \in \nat}$ of measurable maps from $X$ to $\creal$ that is
pointwise monotonic, their pointwise supremum $h$ is measurable, and
$\int_X hd\mu = \sup_{n \in \nat} \int_X h_n d\mu$.  If
${(h_n)}_{n \in \nat}$ is antitonic instead, then a similar theorem
holds provided that $\int_X h_n d\mu < +\infty$ for some $n \in \nat$
(but not in general): the pointwise infimum $h$ is measurable, and
$\int_X hd\mu = \inf_{n \in \nat} \int_X h_n d\mu$.

A practical way of building measures is \emph{Carath\'eodory's measure
  existence theorem}, which is the following.  A \emph{semi-ring}
$\mathcal R$ on $X$ is a collection of subsets of $X$ that is closed
under finite intersections, and such that the complement of every
element of $\mathcal R$ can be written as a finite disjoint union of
elements of $\mathcal R$.  A map $\mu \colon \mathcal R \to \creal$ is called
\emph{$\sigma$-additive}, extending the definition given above, if and
only if for every countable (possibly empty) collection of pairwise
distinct elements $E_n$ of $\mathcal R$ whose union $E$ is also in
$\mathcal R$, $\mu (E) = \sum_n \mu (E_n)$.  Then $\mu$ extends to a
measure on some $\sigma$-algebra containing $\mathcal R$.  A first use
of this theorem is to establish the existence of \emph{Lebesgue measure}
$\lambda$ on $\real$, defined so that $\lambda (]a, b[) = b-a$ for
every open bounded interval $]a, b[$.

A measure $\mu$ on $X$ is \emph{bounded} if and only if
$\mu (X) < +\infty$.  A measure $\mu$ is \emph{$\sigma$-finite} if
there is a sequence
$E_0 \subseteq E_1 \subseteq \cdots \subseteq E_n \subseteq \cdots$ of
measurable subsets of $X$ whose union is $X$ and such that
$\mu (E_n) < +\infty$ for every $n \in \nat$.  A \emph{$\pi$-system}
$\Pi$ on a set $X$ is a family of sets closed under finite
intersections.  If $X$ is a measurable space such that
$\Sigma_X = \Sigma (\Pi)$, any two $\sigma$-finite measures that agree
on $\Pi$ also agree on $\Sigma_X$.  In particular, Lebesgue measure on
$\real$ is uniquely defined by the specification
$\lambda (]a, b[) = b-a$.

\paragraph{Domain theory and topology.}

A \emph{dcpo} is a poset in which every directed family $D$ has a
supremum $\sup D$.  A prime example is $\IRbb$, the poset of closed
intervals $[a, b]$ with $a, b \in \real \cup \{-\infty, +\infty\}$ and
$a \leq b$, ordered by reverse inclusion.  Every directed family
${([a_i, b_i])}_{i \in I}$ has a supremum
$\bigcap_{i \in I} [a_i, b_i] = [\sup_{i \in I} a_i, \inf_{i \in I}
b_i]$.  
Among them, we find the \emph{total}
numbers $a \in \real \cup \{-\infty, +\infty\}$, which are equated
with the maximal elements $[a, a]$ in~$\IRbb$.

Another example is $\creal$, with the usual ordering.  We will also
consider $\Icreal$, the subdcpo of $\IRbb$ consisting of its elements
of the form $[a, b]$ with $a \geq 0$.

We will also write $\leq$ for the ordering on any poset.  In the
example of $\IRbb$, $\leq$ is $\supseteq$.  The \emph{upward closure}
$\upc A$ of a subset $A$ of a poset $X$ is
$\{y \in X \mid \exists x \in A, x \leq y\}$.  The \emph{downward
  closure} $\dc A$ is defined similarly.  A set $A$ is \emph{upwards
  closed} if and only if $A = \upc A$, and \emph{downwards closed} if
and only if $A = \dc A$.  A subset $U$ of a dcpo $X$ is
\emph{Scott-open} if and only if it is upwards closed and, for every
directed family $D$ such that $\sup D \in U$, some element of $D$ is
in $U$ already.  The Scott-open subsets of a dcpo $X$ form its
\emph{Scott topology}.

The \emph{way-below} relation $\ll$ on a poset $X$ is defined by
$x \ll y$ if and only if, for every directed family $D$ with a
supremum $z$, if $y \leq z$, then $x$ is less than or equal to some
element of $D$ already.  We write $\uuarrow x$ for
$\{y \in X \mid x \ll y\}$, and $\ddarrow y$ for
$\{x \in X \mid x \ll y\}$.  A poset $X$ is \emph{continuous} if and
only if $\ddarrow x$ is directed and has $x$ as supremum for every
$x \in X$.  A \emph{basis} $B$ of a poset $X$ is a subset of $X$ such
that $\ddarrow x \cap B$ is directed and has $x$ as supremum for every
$x \in X$.  A poset $X$ is continuous if and only if it has a basis
(namely, $X$ itself).  A poset is \emph{$\omega$-continuous} if and
only if it has a countable basis.  Examples include $\creal$, with any
countable dense subset (with respect to its standard topology), such
as the rational numbers in $\creal$, or the \emph{dyadic numbers}
$k/2^n$ ($k, n \in \nat$); or $\IRbb$ and $\Icreal$, with the basis of
intervals $[a, b]$ where $a$ and $b$ are both dyadic or rational.

We write $\Open X$ for the lattice of open subsets of a topological
space $X$.  This applies to dcpos $X$ as well, which will always be
considered with their Scott topology.  The continuous maps
$f \colon X \to Y$ between two dcpos coincide with the
\emph{Scott-continuous} maps, namely the monotonic (order-preserving)
maps that preserve all directed suprema.  We write $\Lform X$ for the
space of continuous maps from a topological space $X$ to $\creal$, the
latter with its Scott topology, as usual.  Such maps are usually
called \emph{lower semicontinuous}, or \emph{lsc}, in the mathematical
literature.  Note that $\Lform X$, with the pointwise ordering, is a
dcpo.

There are several ways in which one can model probabilistic choice.
The most classical one is through measures.  A popular alternative
used in domain theory is given by \emph{continuous valuations}
\cite{jones89,jones90}.  A continuous valuation is a Scott-continuous
map $\nu \colon \Open X \to \creal$ such that $\nu (\emptyset)=0$
(\emph{strictness}) and, for all $U, V \in \Open X$,
$\nu (U \cup V) + \nu (U \cap V) = \nu (U) + \nu (V)$
(\emph{modularity}).  There is a notion of integral
$\int_{x \in X} h (x) d\nu$, or briefly $\int h d\nu$, for every
$h \in \Lform X$, which can again be defined by a Choquet formula.
The map $h \in \Lform X \mapsto \int h d\nu$ is Scott-continuous and
\emph{linear}.  By definition, a linear map
$G \colon \Lform X \to \creal$ satisfies $G (h+h') = G (h)+G(h')$ and
$G (\alpha.h) = \alpha.G(h)$ for all $\alpha \in \realp$,
$h, h' \in \Lform X$.  Conversely, any Scott-continuous linear map
$G \colon \Lform X \to \creal$ is of the form $h \mapsto \int h d\nu$
for a unique continuous valuation $\nu$, given by
$\nu (U) \eqdef G (\chi_U)$, where $\chi_U$ is the characteristic map
of $U$ ($\chi_U (x) \eqdef 1$ if $x \in U$, $0$ otherwise).

\section{Rags, d-rags and continuous d-rags}
\label{sec:rugs-d-rugs}

\begin{definition}
  \label{defn:rag}
  A \emph{rag} is a tuple $(R, 0, +, 1, \times)$ (or simply $R$) where
  $(R, 0, +)$ is an Abelian monoid, $(R, 1, \times)$ is a monoid, and
  $\times$ distributes over $+$.  An \emph{Abelian rag} is a rag whose
  multiplication $\times$ is commutative.
\end{definition}
A \emph{semi-ring}, or \emph{rig}, is a rag which satisfies the extra
law $0 \times r = r \times 0 = 0$.  $\creal$, for example, is an
Abelian rig, where $+$ and $\times$ are as usual, modulo the
convention that $0 \times (+\infty) = 0$.  We will see that $\Icreal$
is a rag, but not a rig.

We also need some topological structure.
\begin{definition}
  \label{defn:drag}
  A \emph{d-rag} is a rag $R$ together with an ordering that makes it
  a dcpo, in such a way that $+$ and $\times$ are Scott-continuous.
\end{definition}

$\creal$ is a d-rag. In order to turn $\Icreal$ into a d-rag, we
define its $0$ element as $[0, 0]$; addition by
$[a, b] + [c, d] \eqdef [a+c, b+d]$; its $1$ element as $[1, 1]$; and
product by
$[a, b] \times [c, d] \eqdef [a \cdot_\ell c, b \cdot_r d]$.  The
operations $\cdot_\ell$ and $\cdot_r$ are product operations (for the
$\ell$eft and $r$ight part, respectively), and are defined so that
$x \cdot_\ell y$ and $x \cdot_r y$ are equal to the usual product $xy$
unless one of $x$, $y$ is equal to $0$ and the other is equal to
$+\infty$.  We need two distinct, left and right, product operations
in order to ensure Scott-continuity, as we now explain.  We must
define $0 \cdot_\ell (+\infty)$ ($=(+\infty) \cdot_\ell 0$) as $0$,
since $0 \cdot_\ell (+\infty)$ must be equal to
$\sup_{r \in \realp} 0 \cdot_\ell r = 0$.  Symmetrically, we must
define $0 \cdot_r (+\infty)$ ($=(+\infty) \cdot_r 0$) as $+\infty$,
because $0 \cdot_r (+\infty)$ must be equal to
$\inf_{r > 0} r \cdot_r (+\infty) = +\infty$.  With those choices, we
have the following easily checked fact.

\begin{lemma}
  \label{lemma:IR:drag}
  $\Icreal$ is an Abelian d-rag.
\end{lemma}

\section{Continuous $R$-Valuations}
\label{sec:cont-r-valu}

Let $R$ be a fixed Abelian d-rag.  One might be tempted to define
continuous $R$-valuations on a space $X$ as Scott-continuous maps from
$\Open X$ to $R$ satisfying some appropriate forms of strictness and
modularity, but, as we have argued in the introduction, this is
fraught with difficulties when the additive unit is not the least
element of $R$.

Since continuous valuations on $X$ correspond bijectively to linear
Scott-continuous maps from $\Lform X$ to $\creal$, another route is to
define continuous $R$-valuations as certain maps from a variant of
$\Lform X$ to $R$ instead of $\creal$.  As we will see, this leads to
a streamlined theory.

Given any space $X$, let $\Lform^R X$ be the dcpo of all continuous
maps from $X$ to $R$, with the pointwise ordering.  With pointwise
addition and multiplication, $\Lform^R X$ is also an Abelian d-rag.
\begin{definition}[Continuous $R$-valuation]
  \label{defn:Rval}
  A \emph{continuous $R$-valuation} on a space $X$ is a
  Scott-continuous map $\nu \colon \Lform^R X \to R$ that is
  \emph{linear} in the sense that
  $\nu (\mathbf a \times h) = \mathbf a \times \nu (h)$
  (\emph{homogeneity}) and $\nu (h+h') = \nu (h) + \nu (h')$
  (\emph{additivity}) for all $\mathbf a \in R$,
  $h, h' \in \Lform^R X$.  We write $\Val^R X$ for the dcpo of all
  continuous $R$-valuations on $X$, with the pointwise ordering.
\end{definition}
\begin{remark}
  \label{rem:val}
  When $R = \creal$, $\Lform^R X = \Lform X$, so that $\Val^R X$ can
  be equated with the dcpo $\Val X$ of ordinary continuous valuations.
\end{remark}
In order to help understand the definition, it is profitable to use
the integral notation $\int h d\nu$ to mean $\nu (h)$.  Hence
Definition~\ref{defn:Rval} requires that
$\int (\mathbf a\times h) d\nu = \mathbf a \times \int h d\nu$ and
$\int (h+h') d\nu = \int h d\nu + \int h' d\nu$.

Beware that the constant $0$ map from $\Lform^R X$ to $R$ is
\emph{not} a continuous $R$-valuation, unless $R$ is a rig:
homogeneity would imply $0 = \mathbf a \times 0$, which fails in
$\Icreal$, for example.
Also, we do not require $\nu (\mathbf 0)=0$ in
Definition~\ref{defn:Rval}, where $\mathbf 0$ is the constant $0$ map.
This would be a consequence of homogeneity if $R$ were a rig.

Addition and multiplication by scalars in $R$ are defined pointwise on
$\Val^R X$.  This allows us to make sense of the following definition.
\begin{definition}
  \label{defn:simple}
  The \emph{$R$-Dirac mass at $x \in X$} is the continuous
  $R$-valuation $\delta_x \colon h \in \Lform^R X \mapsto h (x) \in R$.
  
  An \emph{elementary $R$-valuation} on $X$ is a continuous
  $R$-valuation of the form $\sum_{i=1}^n r_i \times \delta_{x_i}$,
  where $n \geq 1$, each $r_i$ is in $R$, and mapping each
  $h \in \Lform^R X$ to $\sum_{i=1}^n r_i \times h (x_i)$.

  We write $\Val^R_\fin X$ for the poset of elementary $R$-valuations on
  $X$, and $\Val^R_m X$ for the inductive closure of $\Val^R_\fin X$ in
  $\Val^R X$.  The elements of $\Val^R_m X$ are called the
  \emph{minimal $R$-valuations} on $X$.
\end{definition}
Our definition of the $R$-Dirac mass reads $\int h d\delta_x = h (x)$
in integral notation.

\begin{remark}
  \label{rem:simple}
  A \emph{simple valuation} on $X$ is one of the form
  $\sum_{i=1}^n r_i \delta_{x_i}$, where $n \in \nat$, each point
  $x_i$ is in $X$, and each coefficient $r_i$ is in $\realp$.  While
  continuous valuations can be equated with continuous $R$-valuations
  with $R = \creal$ (Remark~\ref{rem:val}), simple valuations and
  elementary $R$-valuations are closely related but different
  concepts, even when $R = \creal$.  First, $r \delta_x$ is an
  elementary $\creal$-valuation even when $r = +\infty$, but is a
  simple valuation only if $r < +\infty$.  Second, we require
  $n \geq 1$ in the definition of elementary $R$-valuations, but $n$
  can be equal to $0$ in the definition of a simple valuation.  The
  reason why we require $n \geq 1$ is that the constant $0$ map is not
  a continuous $R$-valuation in general, as noticed above.
\end{remark}

The \emph{inductive closure} of a subset $A$ of a dcpo $Z$ is the
smallest subset of $Z$ that contains $A$ and is closed under directed
suprema.  It is obtained by taking all directed suprema of elements of
$A$, all directed suprema of elements obtained in this fashion, and
proceeding this way transfinitely.

A \emph{pointed} dcpo is one with a least element $\bot$.
\begin{proposition}
  \label{prop:V:pointed}
  Let $R$ be an Abelian $d$-rag with a least element $\bot$ that is
  absorbing for multiplication, viz., $\bot \times \mathbf a = \bot$
  for every $\mathbf a \in R$.  For every non-empty space $X$, the
  constant map $\tilde\bot \colon h \in \Lform^R X \mapsto \bot$ is
  the least element of $\Val^R X$, and also of $\Val^R_m X$.
  Thus,  $\Val^R X$ and $\Val^R_m X$ are pointed dcpos.
\end{proposition}
\begin{proof}
  Since $\bot$ is absorbing, $\tilde\bot$ is equal to
  $\bot \times \delta_x$, for any fixed $x \in X$, hence is in
  $\Val^R_m X$.  It is clearly least in $\Val^R X$ and in
  $\Val^R_m X$.  The last claim is obvious.
\end{proof}
Proposition~\ref{prop:V:pointed} applies to the case $R = \creal$,
where the bottom element is $0$, and $0 \times r = r$ for every $r$
(including $+\infty$).  It also applies to the case $R = \Icreal$, where
the bottom element is $[0, +\infty]$, and again $[0, +\infty] \times
[a, b] = [0 \cdot_\ell a, (+\infty) \cdot_r b] = [0, +\infty]$.

  \begin{remark}
    There is a very similar notion of integration of interval-valued
    functions, yielding interval values, due to Edalat
    \cite{edalat09}, which he uses to define interval-valued integrals
    of measurable functions.  The purpose is to set up a computable
    framework for Lebesgue measure and integration theory.  The two
    integrals considered in \cite{edalat09} and in the present paper
    are similar, but different in a subtle way.  There are small
    differences, such as the fact that Edalat allows one to integrate
    functions with values in $\IR$, whereas we only integrate with
    values in $\Icreal$, but the main difference is best illustrated
    by the following example.  Let $\lambda$ be Lebesgue measure on
    $[0, 1]$, and $h_n \colon [0, 1] \to \Icreal$ map every
    $x \in [0, 1/2^n]$ to $[0, \infty]$ and every $x \in ]1/2^n, 1]$
    to $[0, 0]$.  The maps $h_n$ form a chain whose supremum is the
    function $h$ that maps every element of $]0, 1]$ to $[0, 0]$ and
    $0$ to $[0, \infty]$.  Using Edalat's integral, we have
    $\int h_n d\lambda = [0, \infty]$ for every $n \in \nat$, but
    $\int h d\lambda = [0, 0]$.  This shows that Edalat's integral is
    \emph{not} a continuous $\Icreal$-valuation in general.  We will
    propose a way to fix this issue in Section~\ref{sec:meas-as-cont}.
\end{remark}

\section{Monads of continuous $R$-valuations}
\label{sec:monads-continuous-r}

We fix an Abelian d-rag $R$.  We will see that $\Val^R$ and $\Val^R_m$
define strong monads on the category $\Dcpo$ of dcpos and
Scott-continuous maps.  This is essential in describing probabilistic
effects, following Moggi's seminal work \cite{moggi89,moggi91}.  We
use Manes' presentation of monads \cite{Manes:algth}: a monad
$(T, \eta, \_^\dagger)$ on a category $\catc$ is a function $T$
mapping objects of $\catc$ to objects of $\catc$, a collection of
morphisms $\eta_X \colon X \to TX$, one for each object $X$ of
$\catc$, and called the \emph{unit}, and for every morphism
$f \colon X \to TY$, a morphism $f^\dagger \colon TX \to TY$ called
the \emph{extension} of $f$; those are required to satisfy the axioms:
\begin{enumerate}
\item $f^\dagger \circ \eta_X = f$;
\item $\eta_X^\dagger = \identity {TX}$;
\item $(g^\dagger \circ f)^\dagger = g^\dagger \circ f^\dagger$.
\end{enumerate}
Then $T$ extends to an endofunctor, acting on morphisms through
$T f = (\eta_Y \circ f)^\dagger$.
\begin{proposition}
  \label{prop:VX:monad}
  The triple $(\Val^R, \eta, \_^\dagger)$ is a monad on the category
  of dcpos and Scott-continuous maps, where
  $\eta_X \colon X \to \Val^R X$ maps $x$ to $\delta_x$, and for every
  $f \colon X \to \Val^R Y$, $f^\dagger$ is defined by
  $f^\dagger (\nu) (k) \eqdef \nu (\lambda x \in X . f (x) (k))$ for
  every $\nu \in \Val^R X$, and for every $k \in \Lform^R Y$.
\end{proposition}
\begin{proof}
  Verifying that $\eta$ is Scott-continuous is routine.

  Let us look at $f^\dagger$.  For every $k \in \Lform^R Y$, it is
  easy to see that $\lambda x \in X . f (x) (k)$ is Scott-continuous,
  because $f$ is Scott-continuous and directed suprema are computed
  pointwise in $\Val^R Y$.  Hence $\nu (\lambda x \in X . f (x) (k))$
  makes sense.  The map
  $f^\dagger (\nu) \colon k \mapsto \nu (\lambda x \in X . f (x) (k))$
  is also Scott-continuous, since $f (x)$ is Scott-continuous for
  every $x \in X$ and since $\nu$ is itself Scott-continuous.  It is
  easy to see that $f^\dagger (\nu)$ is linear, too, because $f (x)$
  is linear for every $x \in X$, and because $\nu$ is linear.  Hence
  $f^\dagger (\nu)$ is an element of $\Val^R Y$ for every
  $\nu \in \Val^R X$.  Finally, $f^\dagger$ itself is
  Scott-continuous, as one easily checks.

  The monad equations (i), (ii) and (iii) are immediate.
\end{proof}
\begin{fact}
  \label{fact:Vf}
  The $\Val^R$ functor acts on morphisms by
  $\Val^R (f) (\nu) (k) = \nu (k \circ f)$.
\end{fact}
In integral notation, this means $\nu' \eqdef \Val^R (f) (\nu)$
satisfies $\int k d\nu' = \int (k \circ f) d\nu$.  This is a formula
that is typical of the image measure of $\nu$ by $f$, where $\nu$ is a
measure.  We may think of $\Val^R (f) (\nu)$ as the image of the
continuous $R$-valuation $\nu$ by $f$.


We will now show that $\Val^R_m$ defines a submonad of $\Val^R$.  To
this end, we need to know more about inductive closures.  A
\emph{d-closed subset} of a dcpo $Z$ is a subset $C$ such that the
supremum of every directed family of elements of $C$, taken in $Z$, is
in $C$.  The d-closed subsets form the closed subsets of a topology
called the \emph{d-topology} \cite[Section~5]{keimel08}, and the
inductive closure of a subset $A$ coincides with its \emph{d-closure}
$cl_d (A)$, namely its closure in the d-topology.

We note that every Scott-continuous map is continuous with respect to
the underlying d-topologies.  This is easily checked, or see
\cite[Lemma~5.3]{keimel08}.  In particular:
\begin{fact}
  \label{fact:cld}
  For every Scott-continuous map $f \colon \Val^R X \to \Val^R Y$, for
  every $A \subseteq \Val^R X$, $f (cl_d (A)) \subseteq cl_d (f (A))$.
\end{fact}

\begin{lemma}
  \label{lemma:VXm:drag}
  For every space $X$, $\Val^R_m X$ is closed under addition and
  multiplication by elements of $R$, as computed in the larger space
  $\Val^R X$.
\end{lemma}
\begin{proof}
  Let us deal with addition.  Multiplication is similar.
  
  For every elementary $R$-valuation $\mu$, the map
  $f_\mu \colon \nu \in \Val^R X \mapsto \mu + \nu$ is
  Scott-continuous, and maps elementary $R$-valuations to elementary
  $R$-valuations.  By Fact~\ref{fact:cld} with
  $A \eqdef \Val^R_\fin X$, $f_\mu$ maps all elements of
  $cl_d (A) = \Val^R_m X$ to
  $cl_d (f_\mu (A)) \subseteq cl_d (\Val^R_\fin X) = \Val^R_m X$.

  It follows that for every minimal $R$-valuation $\nu$, the map
  $g \colon \mu \in \Val^R X \mapsto \mu+\nu = f_\mu (\nu)$ maps
  elementary $R$-valuations to minimal $R$-valuations.  We observe that
  $g$ is also Scott-continuous.  By Fact~\ref{fact:cld} with the same
  $A$ as above, $g$ maps all elements of $cl_d (A) = \Val^R_m X$ to
  $cl_d (g (A)) \subseteq cl_d (\Val^R_m X) = \Val^R_m X$.  Hence, for
  every $\nu \in \Val^R_m X$, for every $\mu \in \Val^R_m X$, $\mu +
  \nu$ is in $\Val^R_m X$.
\end{proof}

\begin{lemma}
  \label{lemma:VXm:dagger}
  For any Scott-continuous map $f \colon X \to \Val^R_m Y$,
  $f^\dagger$ is a Scott-continuous map from $\Val^R_m X$ to $\Val^R_m Y$.
\end{lemma}
\begin{proof}
  The only challenge is to show that, for every $\nu \in \Val^R_m X$,
  $f^\dagger (\nu)$ is in $\Val^R_m Y$.  Scott-continuity follows
  from the fact that $f^\dagger$ is Scott-continuous from $\Val^R X$
  to $\Val^R Y$.

  For every $\nu \eqdef \sum_{i=1}^n r_i \times \delta_{x_i}$ in
  $\Val^R_\fin X$ ($n \geq 1$), $f^\dagger (\nu)$ is the continuous
  $R$-valuation $\sum_{i=1}^n r_i \times f (x_i)$: for every
  $k \in \Lform^R Y$,
  $f^\dagger (\nu) (k) = \nu (\lambda x \in X . f (x) (k)) =
  \sum_{i=1}^n r_i \times \allowbreak f (x_i) (k) = (\sum_{i=1}^n r_i
  \times f (x_i)) (k)$.
  By Lemma~\ref{lemma:VXm:drag}, and since $f (x_i)$ is in $\Val^R_m
  Y$ for each $i$, $f^\dagger (\nu)$ is in $\Val^R_m Y$ as well.

  Hence $f^\dagger$ maps $\Val^R_\fin X$ to $\Val^R_m Y$.  Using
  Fact~\ref{fact:cld} with $A \eqdef \Val^R_\fin X$,
  $f^\dagger (cl_d (A)) = f^\dagger (\Val^R_m X)$ is included in
  $cl_d (f^\dagger (A)) \subseteq cl_d (\Val^R_m Y) = \Val^R_m Y$.
\end{proof}

We observe that $\eta_X (x) = \delta_x$ is in
$\Val^R_\fin X \subseteq \Val^R_m X$ for every dcpo $X$, and every
$x \in X$, whence the following.
\begin{proposition}
  \label{prop:VX:monadm}
  The triple $(\Val^R_m, \eta, \_^\dagger)$ is a monad on the category
  of dcpos and Scott-continuous maps.
\end{proposition}

A tensorial strength for a monad $(T, \eta, \_^\dagger)$ is a
collection $t$ of morphisms
$t_{X,Y} \colon X \times TY \to T (X \times Y)$, natural in $X$ and
$Y$, satisfying certain coherence conditions (which we omit, see
\cite{moggi91}.)  We then say that $(T, \eta, \_^\dagger, t)$ is a
\emph{strong} monad.  We will satisfy ourselves with the following
result.  By \cite[Proposition~3.4]{moggi91}, in a category with finite
products and enough points, if one can find morphisms $t_{X,Y}$ for
all objects $X$ and $Y$ such that
$t_{X,Y} \circ \langle x, \nu\rangle = T (\langle x \circ !, \identity
{Y} \rangle) \circ \nu$, where $!\colon Y\to 1$ is the unique morphism from $Y$ to the terminal object,
then the collection of those morphisms is the
unique tensorial strength.  The category of dcpos has finite products,
and has enough points, and specializing this to the $\Val^R$ monad, we
obtain the following.
\begin{lemma}
  \label{lemma:VX:t}
  The maps $t_{X,Y} \colon X \times \Val^R Y \to \Val^R (X \times Y)$
  defined by
  $t_{X,Y} (x, \nu) \eqdef \lambda h \in \Lform^R (X \times Y) . \nu
  (\lambda y \in Y . h (x, y))$ define the unique tensorial strength
  for the monad $(\Val^R, \eta, \_^\dagger)$.
\end{lemma}
\begin{proof}
  The previous observation shows that we must define $t_{X,Y}$ by
  $t_{X,Y} (x, \nu) \eqdef \Val^R (\lambda y \in Y . (x, y)) (\nu)$.
  By Fact~\ref{fact:Vf}, the latter is equal to
  $\lambda h \in \Lform^R (X \times Y) . \nu (h \circ \lambda y \in Y
  . (x, y)) = \lambda h \in \Lform^R (X \times Y) . \nu (\lambda y \in
  Y . h (x, y))$.

  It is enough to check that $t_{X,Y}$ is Scott-continuous.  This
  follows from the fact that application (of $\nu$ to $\lambda y \in
  Y.  h (x, y))$) is Scott-continuous.
\end{proof}
In integral notation, $t_{X,Y} (x, \nu)$ is the continuous
$R$-valuation $\nu'$ such that $\int h d\nu' = \int h (x, \_)d\nu$ for
every $h \in \Lform^R (X \times Y)$.

For every $\nu \in \Val^R_m Y$, for every $x \in X$, $t_{X,Y} (x,
\nu)$ is equal to $\Val^R (\lambda y \in Y . (x, y)) (\nu) = (\eta
\circ \lambda y \in Y . (x,y))^\dagger (\nu)$.  By
Lemma~\ref{lemma:VXm:dagger}, this is an element of $\Val^R_m (X
\times Y)$.  It follows:
\begin{proposition}
  \label{prop:VmX:t}
  $(\Val^R, \eta, \_^\dagger, t)$ and
  $(\Val^R_m, \eta, \_^\dagger, t)$ are strong monads on $\Dcpo$.
\end{proposition}
We now show that $\Val^R_m$ is a \emph{commutative} monad.  The
corresponding result is unknown for $\Val^R$, even when $R = \creal$.

Given a tensorial strength $t$, there is a dual tensorial strength
$t'$, where $t'_{X,Y} \colon TX \times Y \to T (X \times Y)$.  Here
$t'_{X,Y} (\mu, y) = \lambda h \in \Lform^R (X \times Y) . \mu
(\lambda x \in X . h (x, y))$.  We can then define two morphisms from
$TX \times TY$ to $T (X \times Y)$, namely
${t'}^\dagger_{X,Y} \circ t_{TX,Y}$ and
$t^\dagger_{X,Y} \circ t'_{X,TY}$.  The monad $T$ is
\emph{commutative} when they coincide.
\begin{lemma}
  \label{lemma:f=g:cld}
  Two morphisms $f, g \colon X \to Y$ in $\Dcpo$ that coincide on
  $A \subseteq X$ also coincide on $cl_d (A)$.
\end{lemma}
\begin{proof}
  Let $B \eqdef \{x \in X \mid f (x)=g (x)\}$.  Since $f$ and $g$
  preserve directed suprema, $B$ is d-closed.  By assumption, $A$ is
  included in $B$, so $B$ also contains $cl_d (A)$.
\end{proof}

\begin{proposition}
  \label{prop:VmX:comm}
  Let $X$, $Y$ be two dcpos.  The maps
  $t^\dagger_{X,Y} \circ t'_{X,TY}$ and
  ${t'}^\dagger_{X,Y} \circ t_{TX,Y}$ coincide on those pairs
  $(\mu, \nu) \in \Val^R X \times \Val^R Y$ such that
  $\mu \in \Val^R_m X$ or $\nu \in \Val^R_m Y$.
\end{proposition}
\begin{proof}
  We prove the claim when $\mu \in \Val^R_m X$.  The case where $\nu
  \in \Val^R_m Y$ is symmetric.

  In the sequel, $h$ ranges over $\Lform^R (X \times \Val^R Y)$, $k$
  over $\Lform^R (X \times Y)$, $x$ over $X$, $y$ over $Y$, and $\nu$
  over $\Val^R Y$.  For all $\mu \in \Val^R X$ and $\nu \in \Val^R Y$,
  we verify that:
  \begin{align*}
    (t^\dagger_{X,Y} \circ t'_{X,TY}) (\mu, \nu)
    & = \lambda k .
      \mu (\lambda x . \nu (\lambda y . k (x, y)))
      \\
    ({t'}^\dagger_{X,Y} \circ t_{TX,Y}) (\mu, \nu)
    & = \lambda k .
      \nu (\lambda y . \mu (\lambda x . k (x, y))).
  \end{align*}
  Those two quantities are equal when $\mu$ is an elementary
  $R$-valuation $\sum_{i=1}^m r_i \times \delta_{x_i}$ ($m \geq 1$),
  since the first one is equal to
  $\lambda k.  \sum_{i=1}^m r_i \times \nu (\lambda y . k (x_i, y))$,
  the second one is equal to
  $\lambda k .  \nu (\lambda y . \sum_{i=1}^m r_i \times k (x_i, y))$,
  and since $\nu$ is linear.


  For fixed $\nu \in \Val^R Y$, the maps
  $f \colon \mu \in \Val^R X \mapsto (t^\dagger_{X,Y} \circ t'_{X,TY})
  (\mu, \nu)$ and
  $g \colon \mu \in \Val^R X \mapsto ({t'}^\dagger_{X,Y} \circ
  t_{TX,Y}) (\mu, \nu)$ therefore coincide on $\Val^R_\fin X$.  They
  are both Scott-continuous, since $t^\dagger_{X,Y} \circ t'_{X,TY}$
  and ${t'}^\dagger_{X,Y} \circ t_{TX,Y}$ are.  By
  Lemma~\ref{lemma:f=g:cld}, they must coincide on the d-closure of
  $\Val^R_\fin X$, which is $\Val^R_m X$ by definition.
\end{proof}

\begin{corollary}
  \label{corl:VmX:comm}
  $(\Val^R_m, \eta, \_^\dagger, t)$ is a commutative monad.
\end{corollary}


The equality of Proposition~\ref{prop:VmX:comm} is that, for every
$\mu \in \Val^R_m X$ and for every $\nu \in \Val^R_m Y$, for every
$k \in \Lform^R (X \times Y)$,
$\mu (\lambda x \in X . \nu (\lambda y \in Y . k (x, y))) = \nu
(\lambda y \in Y . \mu (\lambda x \in X . k (x, y)))$.  In integral
notation,
\[
  \int_{x \in X} \left(\int_{y \in Y} k (x, y) d\nu \right) d\mu
  = \int_{y \in Y} \left(\int_{x \in X} k (x, y) d\mu\right)d\nu,
\]
which we recognize as the integral permutation property, obtained
in the classical measure-theoretic case as a consequence of Fubini's
theorem.

Fubini's theorem is more general, and states the existence of a
product measure.  A similar fact follows from the above results, as
noticed by Kock \cite{Kock:monad:comm}.  We write $\otimes$ for the
morphism ${t'}^\dagger_{X,Y} \circ t_{TX,Y}$ and
$t^\dagger_{X,Y} \circ t'_{X,TY}$ from $\Val^R_m X \times \Val^R_m Y$
to $\Val^R_m (X \times Y)$, as with any commutative monad
\cite[Section~5]{Kock:monad:comm}.  Then, for all $\mu \in \Val^R_m X$
and $\nu \in \Val^R_m Y$, $\otimes (\mu, \nu)$, which we prefer to
write as $\mu \otimes \nu$, is in $\Val^R_m (X \times Y)$, and by
definition
$(\mu \otimes \nu) (k) = \mu (\lambda x \in X . \nu (\lambda y \in Y
. k (x, y))) = \nu (\lambda y \in Y . \mu (\lambda x \in X . k (x,
y)))$.  In integral notation, we obtain the following form of Fubini's
theorem:
\[
  \int_{(x, y) \in X \times Y} k (x, y) d(\mu \otimes \nu)
  = \int_{x \in X} \left(\int_{y \in Y} k (x, y) d\nu \right) d\mu
  = \int_{y \in Y} \left(\int_{x \in X} k (x, y) d\mu\right)d\nu,
\]
for all $\mu \in \Val^R_m X$, $\nu \in \Val^R_m Y$, and
$k \in \Lform^R (X \times Y)$.  As an additional benefit, we obtain (for
free!) that the map
$\otimes \colon (\mu, \nu) \mapsto \mu \otimes \nu$ is
Scott-continuous.

\begin{remark}
  \label{rem:fubini}
  A similar Fubini-like theorem was already obtained by Jones
  \cite{jones90} for arbitrary (subprobability) continuous valuations,
  but in the setting of continuous dcpos only.  Whether the
  Fubini-like formula above holds for every pair of continuous
  valuations $\mu$ and $\nu$ on arbitrary dcpos is unknown.  We note
  that the problem would be easily solved if all continuous valuations
  were minimal, but that is not the case, as is shown in the paper~\cite{goubault-jia-2021}.
\end{remark}

\section{Continuous $R$-valuations and measures I: A brief viewpoint}
\label{sec:cont-r-valu-1}

We look at the special cases of continuous $R$-valuations when $R$ is
$\creal$ or $\Icreal$, and we investigate their relations to measures.

When $R = \creal$, this is simple: as noticed in Remark~\ref{rem:val},
we can equate continuous $R$-valuations with continuous valuations.
Next, continuous valuations and measures are pretty much the same
thing on $\omega$-continuous dcpos, namely on continuous dcpos with a
countable basis.  This holds more generally on de Brecht's
quasi-Polish spaces \cite{deBrecht:qPolish}, a class of spaces that
contains not only the $\omega$-continuous dcpos from domain theory but
also the Polish spaces from topological measure theory.  One can see
this as follows.  In one direction, every measure $\mu$ on a
hereditarily Lindel\"of space $X$ is $\tau$-smooth
\cite[Theorem~3.1]{Adamski:measures}, meaning that its restriction to
the lattice of open subsets of $X$ is a continuous valuation.  A
hereditarily Lindel\"of space is a space whose subspaces are all
Lindel\"of, or equivalently a space in which every family of open sets
contains a countable subfamily with the same union.  Every
second-countable space is hereditarily Lindel\"of, and that includes
all quasi-Polish spaces.  In the other direction, every continuous
valuation on an LCS-complete space extends to a Borel measure
\cite[Theorem~1.1]{DGJL-isdt19}.  An \emph{LCS-complete} space is a
space that is homeomorphic to a $G_\delta$ subset of a locally compact
sober space.  Every quasi-Polish space is LCS-complete; in fact, the
quasi-Polish spaces are exactly the second-countable LCS-complete
spaces \cite[Theorem~9.5]{DGJL-isdt19}.

\begin{remark}
  \label{rem:ext:unique}
  A continuous valuation $\mu$ on an LCS-complete space $X$ may extend
  to more than one Borel measure.  However, the extension is unique
  if $\mu$ is \emph{$\sigma$-finite}, namely if there is a monotone
  sequence
  $U_0 \subseteq U_1 \subseteq \cdots \subseteq U_n \subseteq \cdots$
  of open subsets of $X$ whose union is the whole of $X$, and such
  that $\mu (U_n) < +\infty$ for every $n \in \nat$.  Indeed, any
  extension of $\mu$ will be $\sigma$-finite in the usual sense.  We
  conclude since any two $\sigma$-finite measures that agree on all
  open sets (which form a $\pi$-system) must agree on the Borel
  $\sigma$-algebra.
\end{remark}

We now look in detail at the more complex case $R = \Icreal$.  We use
the following notation.  Given any element $\mathbf x$ of $\Icreal$ or
of $\IRbb$, we write $x^-$ and $x^+$ for its endpoints, viz.,
$\mathbf x = [x^-, x^+]$.  Every map $h \colon X \to \IRbb$ defines
two maps $h^-, h^+ \colon X \to \real \cup \{-\infty, +\infty\}$ by
$h^- (x) \eqdef h (x)^-$ and $h^+ (x) \eqdef h (x)^+$.  Given two maps
$f, g \colon X \to \real \cup \{-\infty, +\infty\}$ such that
$f \leq g$, we write $[f, g]$ for the function that maps $x$ to
$[f (x), g (x)]$.  Note that, given any map $f$ in $\Lform X$, the map
$[f, +\infty.\mathbf 1] \colon x \mapsto [f (x), +\infty]$ is in
$\Lform^{\Icreal} X$.  (We will write $r.\mathbf 1$ for the constant
function with value $r$, in order to distinguish it from the scalar
value $r$.)  Given any $\Icreal$-continuous valuation $F$ on a space
$X$, we also define $F^- (h)$ as $F (h)^-$ and $F^+ (h)$ as $F (h)^+$,
for every $h \in \Lform^{\Icreal} X$.
\begin{lemma}[The view from the left I]
  \label{lemma:Rcont->cont}
  Let $X$ be any topological space.  For every continuous
  $\Icreal$-valuation $F$ on $X$, for every
  $h \in \Lform^{\Icreal} X$, $F^- (h)$ only depends on $h^-$, not on
  $h^+$.  Moreover, there is a unique continuous valuation $\nu_F$ on
  $X$ such that, for every $h \in \Lform^{\Icreal} X$,
  \begin{align*}
    F^- (h) = \int_X h^- d\nu_F.
  \end{align*}
\end{lemma}
\begin{proof}
  For the first part, it suffices to show that
  $F^- (h) = F^- ([h^-, +\infty.\mathbf 1])$.  We note that the bottom
  element $[0, +\infty]$ of $\Icreal$ is multiplicatively absorbing:
  for every $\mathbf x \in \Icreal$,
  $[0, +\infty] \times \mathbf x = [0, +\infty]$.  It follows that
  \begin{align*}
    F ([0.\mathbf 1, +\infty.\mathbf 1])
    & = F ([0, +\infty] \times [0.\mathbf 1, +\infty.\mathbf 1]) \\
    & = [0, +\infty] \times F ([0.\mathbf 1, +\infty.\mathbf 1])
      = [0, +\infty].
  \end{align*}
  Next, $[0, +\infty]$ satisfies the following partial absorption law
  for addition: $\mathbf x \in \Icreal$,
  $[0, +\infty] + \mathbf x = [x^-, +\infty]$.  Therefore,
  \begin{align*}
    F ([h^-, +\infty.\mathbf 1])
    & = F (h + [0.\mathbf 1, +\infty.\mathbf 1]) \\
    & = F (h) + F ([0.\mathbf 1, +\infty.\mathbf 1]) \\
    & = F (h) + [0, +\infty] & \text{by our previous result} \\
    & = [F^- (h), +\infty].
  \end{align*}
  It follows that $F^- (h) = F^- ([h^-, +\infty.\mathbf 1])$, and the
  right-hand side does not depend on $h^+$.

  In order to show the second part of the lemma, it suffices to
  observe that the map $f \mapsto F^- ([f, +\infty.\mathbf 1])$ is
  linear and Scott-continuous, and is therefore the integral
  functional of a unique continuous valuation $\nu_F$.
\end{proof}
It follows that, for every $f \in \Lform X$,
$\int_X f d\nu_F = F^- ([f, +\infty.\mathbf 1])$.  In particular, for
every open subset $U$ of $X$,
$\nu_F (U) = F^- ([\chi_U, +\infty.\mathbf 1])$.

There are many ways in which we can reconstruct a continuous
$\Icreal$-valuation from a continuous valuation, and here is the
simplest of all.
\begin{lemma}[The view from the left II]
  \label{lemma:cont->Rcont}
  Let $X$ be any topological space.  For every continuous valuation
  $\nu$ on $X$, there is a smallest continuous $\Icreal$-valuation $F$
  such that $\nu_F=\nu$.  For every $h \in \Lform^{\Icreal} X$,
  \begin{align*}
    F (h) & \eqdef \left[\int_X h^- d\nu, +\infty\right].
  \end{align*}
\end{lemma}
For the continuous $\Icreal$-valuation just given, the view from the
right, namely $F^+ (h)$, is the constant $+\infty$, for every
integrand $h$, including for the constant zero map.  This cannot be
the integral of $h$ with respect to any measure, since the integral of the
zero map is always zero, with respect to any continuous valuation or
measure.

One possible view of continuous $\Icreal$-valuations $F$ is that of
the specification of some unknown measure.  $F^-$ gives a continuous
valuation that is a lower bound on that measure, while $F^+$ measures
how precise that specification is.  In this setting, the continuous
$\Icreal$-valuation $F$ built in Lemma~\ref{lemma:cont->Rcont} is the
least precise specification for $\nu$.  

On more special topological spaces, we will see that every measure
has a much more precise specification, and that it is minimal.


\section{Continuous $R$-valuations and measures II: Measures as continuous $\Icreal$-valuations}
\label{sec:meas-as-cont}

We will see that every non-zero, bounded $\tau$-smooth measure $\mu$
on a coherent topological space $X$ gives rise to a continuous
$\Icreal$-valuation in a natural way.  (A measure $\mu$ on $X$ is
\emph{bounded} if $\mu (X) < \infty$, and we recall that it is
$\tau$-smooth if and only if it restricts to a continuous valuation on
$\Open X$.)  As a first step, we need to define integrals of functions
with values in $\creal$, not just $\Rp$, as is done classically.

More precisely, given a measure $\mu$ on a topological space $X$ (with
its Borel $\sigma$-algebra), we can define the Lebesgue integral
$\int_{x \in X} f (x) d\mu$ of any measurable map
$f \colon X \to \Rp$.  We extend this definition to measurable maps
$f$ from $X$ to $\creal$.  Just as with multiplication in rags, this
comes in two flavors.

Perhaps the most natural extension is:
\begin{align}
  \label{eq:int-}
  \int_{x \in X}^- f (x) d\mu & \eqdef \dsup_{r \in \Rp} \int_{x \in
                                X} \min (f (x), r) d\mu.
\end{align}
It is known that the Lebesgue integral, as used on the right of
(\ref{eq:int-}) can be defined through the following, so-called
Choquet formula \cite[Chapter~VII, Section~48.1, p. 265]{choquet53}:
\begin{align}
  \label{eq:choquet}
  \int_{x \in X} f (x) d\mu & = \int_0^\infty \mu (f^{-1} (]t,
                              \infty])) dt
\end{align}
where the integral on the right is now an indefinite Riemann integral.
As a consequence, and since $(\min (f (\_), r)^{-1} (]t, \infty])$ is
empty for every $t \geq r$, and equal to $f^{-1} (]t, \infty])$ for
every $t < r$, we can rewrite (\ref{eq:int-}) as:
\begin{align}
  \label{eq:int-:choquet}
  \int_{x \in X}^- f (x) d\mu & = \dsup_{r \in \Rp} \int_0^r \mu (f^{-1} (]t,
                              \infty])) dt.
\end{align}

We observe that this lower integral is linear and $\omega$-continuous
(by the monotone convergence theorem).  It also commutes with the
product structure of the d-rag $\creal$.  We also note the following
change of variable formula, for future reference:
\begin{align}
  \label{eq:chgvar}
  \int_{y \in Y}^- f (y) dj[\mu] & = \int_{x \in X}^- f (j (x)) d\mu,
\end{align}
for every measurable map $j \colon X \to Y$, for every measurable
map $f \colon Y \to \creal$, for every measure $\mu$ on $X$, and where
$j [\mu]$ is the \emph{image measure}, defined by $j [\mu] (E) \eqdef
\mu (j^{-1} (E))$.  This is an obvious consequence of~(\ref{eq:int-:choquet}).

A function $f \colon X \to \creal$ is lower semicontinuous if and only
if it is continuous from $X$ to $\creal$ with the Scott topology;
equivalently, for every $r \in \Rp \diff \{0\}$,
$f^{-1} (]r, \infty])$ is open in $X$.  Every lower semicontinuous
function is measurable.

\begin{lemma}
  \label{lemma:int-}
  The lower integral (\ref{eq:int-}) is:
  \begin{enumerate}
  \item additive: for all measurable maps $f$, $g$ from $X$ to $\creal$,
    $\int_{x \in X}^- (f (x) + g (x)) d\mu = \int_{x \in X}^- f (x)
    d\mu + \int_{x \in X}^- g (x) d\mu$;
  \item $\cdot_\ell$-homogeneous: for every measurable map
    $f \colon X \to \creal$, for every $a \in \creal$,
    $a \cdot_\ell \int_{x \in X}^- f (x) d\mu = \int_{x \in X}^- (a
    \cdot_\ell f (x)) d\mu$;
  \item $\omega$-continuous: for every monotonic sequence
    ${(f_n)}_{n \in \nat}$ of measurable maps from $X$ to $\creal$,
    $\int_{x \in X}^- \dsup_{n \in \nat} f_n (x) d\mu = \dsup_{n \in
      \nat} \int_{x \in X}^- f_n (x) d\mu$;
  \item Scott-continuous on lower semicontinuous maps, provided that
    $\mu$ is $\tau$-smooth: for every directed family
    ${(f_i)}_{i \in I}$ of lower semicontinuous maps from $X$ to
    $\creal$,
    $\int_{x \in X}^- \dsup_{i \in I} f_i (x) d\mu = \dsup_{i \in I}
    \int_{x \in X}^- f_i (x) d\mu$.
  \end{enumerate}
\end{lemma}
\begin{proof}
  1. This follows from the additivity of Lebesgue integral of
  $\Rp$-valued functions and the inequalities
  $\min (f (x) + g (x), r) \leq \min (f (x), r) + \min (g (x), r) \leq
  \min (f (x) + g (x), 2r)$.

  2. For every $a \in \Rp \diff \{0\}$,
  \begin{align*}
    a \cdot_\ell \int_{x \in X}^- f (x) d\mu
    & = \dsup_{r \in \Rp} \int_{x \in X} a \min (f (x), r) d\mu \\
    & = \dsup_{r' \in \Rp} \int_{x \in X} \min (a \cdot_\ell f (x),
      r') d\mu & \text{by letting }r' \eqdef ar \\
    & = \int_{x \in X}^- (a \cdot_\ell f (x)) d\mu.
  \end{align*}
  When $a=0$, $a \cdot_\ell \int_{x \in X}^- f (x)
  d\mu$ is equal to $0$ by definition, and $\int_{x \in X}^- (a
  \cdot_\ell f (x)) d\mu = \int_{x \in X} 0 d\mu = 0$.
  
  The new, key case is when $a=\infty$.  We split this into two
  subcases.  If $\mu (f^{-1} (]0, \infty]))=0$, namely if $f$ is
  $\mu$-a.e.\ zero, then $\min (f (\_), r)$ and
  $\min (\infty \cdot_\ell f (\_), r)$ are also $\mu$-a.e.\ zero, so
  $\infty \cdot_\ell \int_{x \in X}^- f (x) d\mu$ and
  $\int_{x \in X}^- (\infty \cdot_\ell f (x)) d\mu$ are both equal to
  $0$.  If $\mu (f^{-1} (]0, \infty])) > 0$, then
  $\mu (f^{-1} (]r, \infty])) > 0$ for some $r \in \Rp \diff \{0\}$,
  since
  $\mu (f^{-1} (]0, \infty])) = \dsup_{q \in \rat_+ \diff \{0\}} \mu
  (f^{-1} (]q, \infty]))$.  It follows that
  $\int_{x \in X}^- f (x) d\mu \geq r \mu (f^{-1} (]r, \infty])) > 0$,
  so $\infty \cdot_\ell \int_{x \in X}^- f (x) d\mu = \infty$, while
  $\int_{x \in X}^- \infty \cdot_\ell f (x) d\mu \geq \int_{x \in X}
  \min (\infty \cdot_\ell f (x), r) d\mu \geq r \mu (f^{-1} (]0,
  \infty]))$ for every $r \in \Rp \diff \{0\}$, so that
  $\int_{x \in X}^- \infty \cdot_\ell f (x) d\mu = \infty$ as well.

  3. We use the monotone convergence theorem, and the fact that
  $\dsup_{n \in \nat} \min (f_n (x), r) = \min (\dsup_{n \in \nat} f_n
  (x), r)$ for all $x \in X$ and $r \in \Rp$.

  4.  Since Riemann integration of non-increasing maps from $\creal$
  to $\creal$ is Scott-continuous (see for example Lemma~4.2 in
  \cite{tix95}), it follows from (\ref{eq:int-:choquet}) that the
  lower integral $\int_{x \in X} f (x) d\mu$ is Scott-continuous in
  the lower semicontinuous map $f$, provided that $\mu$ is
  $\tau$-smooth.
\end{proof}

We also consider the following, upper integral.  This will really only
make sense when the integrated function $f$ is upper semicontinuous,
namely when for every $r \in \Rp \diff \{0\}$, $f^{-1} ([0, r[)$ is
open in $X$; and when the measure $\mu$ is non-zero
($\mu (X) \neq 0$), and $\tau$-smooth.

A \emph{support} of a measure $\mu$ on $X$ is any set $E$ such that,
for all measurable subsets $A$ and $B$ of $X$ such that $A \cap E = B
\cap E$, $\mu (A) = \mu (B)$.  When $E$ is itself measurable, this is
equivalent to requiring $\mu (E) = \mu (X)$, and when $\mu$ is
additionally bounded (i.e., $\mu (X) < \infty$), this is equivalent to
$\mu (X \diff E) = 0$.  We sometimes say that $\mu$ is \emph{supported
  on $E$} to mean that $E$ is a support of $\mu$.

For every $\tau$-smooth measure $\mu$ on $X$, the intersection of all
closed supports of $\mu$ is again a closed support of $\mu$: this
smallest closed support will be denote as $\supp \mu$.  But beware
that there might be smaller (non-closed) supports.  For example,
$\supp (\delta_x)$ is equal to the closure $\dc x$ of the point $x$,
but $\{x\}$ is a smaller (non-closed) support.  Note that $\{x\}$ is
the intersection of the compact saturated set $\upc x$, which happens
to be a support of $\mu = \delta_x$, with $\supp (\mu)$.

In general, not all compact saturated sets $Q$ are measurable, so we
will restrict to \emph{measurable} compact saturated subsets in the
sequel.  The intersection of two supports $E$ and $E'$ may also fail
to be a support, but if one of them (say $E'$) is measurable, then
$E \cap E'$ is also a support.  (Indeed, let $A$, $B$ be measurable
such that $A \cap (E \cap E') = B \cap (E \cap E')$.  Then
$(A \cap E') \cap E = (B \cap E') \cap E$, and since $E$ is a support
of $\mu$, $\mu (A \cap E') = \mu (B \cap E')$.  Since $E'$ is a
support of $\mu$, and since $A$ and $A \cap E'$ have the same
intersection with $E'$, $\mu (A \cap E') = \mu (A)$, and similarly
$\mu (B \cap E') = \mu (B)$.  Therefore $\mu (A) = \mu (B)$.)


For every $\tau$-smooth measure $\mu$ on $X$, we say that a measurable
map $f \colon X \to \creal$ is \emph{$\mu$-bounded} if and only if
there is a measurable compact saturated support $Q$ of $\mu$ such that
$f$ is bounded on $Q \cap \supp \mu$, namely if
$\sup_{x \in Q \cap \supp \mu} f (x) < \infty$.  We will also say that
$Q$ is a \emph{witness} of $\mu$-boundedness of $f$, or that $f$ is
$\mu$-bounded, \emph{witnessed by $Q$}, in that case.  We say that $f$
is \emph{$\mu$-unbounded} if it is not $\mu$-bounded.

\begin{lemma}
  \label{lemma:Qsupp}
  For every non-zero $\tau$-smooth measure $\mu$ on a topological
  space $X$, for every compact saturated support $Q$ of $\mu$,
  $Q \cap \supp \mu$ is non-empty.
\end{lemma}
\begin{proof}
  Otherwise, $\supp \mu$ and the empty set have the same intersection
  with $Q$, and since $Q$ is a support of $\mu$, we would have
  $\mu (\supp \mu) = \mu (\emptyset) = 0$.  Since $\supp \mu$ is a
  measurable support of $\mu$, $\mu (\supp \mu) = \mu (X)$, and
  therefore we would have $\mu (X)=0$, contradicting the fact that
  $\mu$ is non-zero.
\end{proof}

\begin{lemma}
  \label{lemma:max}
  For every non-zero $\tau$-smooth measure $\mu$ on a topological
  space $X$, and for every upper semicontinuous map
  $f \colon X \to \creal$, for every compact saturated
  support $Q$ of $\mu$, there is a point $x \in Q \cap \supp \mu$ such
  that $f (x) = \sup_{x \in Q \cap \supp \mu} f (x)$.
\end{lemma}
\begin{proof}
  Every upper semicontinuous $\creal$-valued function $f$ reaches its
  maximum on any non-empty compact set $K$.  Here is a quick proof:
  let $a \eqdef \sup_{x \in K} f (x)$, and assume that $f (x) < a$ for
  every $x \in K$.  The open sets $f^{-1} ([0, r[)$ with
  $r \in [0, a[$ form an open cover of $K$.  We extract a finite
  subcover $f^{-1} ([0, r[)$, where $r$ ranges over some finite set
  $A$ of numbers strictly below $a$.  This implies that, for every
  $x \in K$, $f (x) < r$ for some $r \in A$, so that
  $a = \sup_{x \in K} f (x) < \max A < a$, a contradiction.
  
  We now apply this to $K \eqdef Q \cap \supp \mu$, which is non-empty
  by Lemma~\ref{lemma:Qsupp}.
\end{proof}

\begin{corollary}
  \label{corl:muunbound}
  For every non-zero $\tau$-smooth measure $\mu$ on a topological
  space $X$, for every upper semicontinuous map
  $f \colon X \to \creal$, $f$ is $\mu$-unbounded if and only if for
  every measurable compact saturated support $Q$ of $\mu$, there is a
  point $x \in Q \cap \supp \mu$ such that $f (x) = \infty$.  \qed
\end{corollary}

This being done, for a $\tau$-smooth measure $\mu$ and an upper
semicontinuous map $f \colon X \to \creal$, we define:
\begin{align}
  \label{eq:int+}
  \int_{x \in X}^+ f (x) d\mu
  & \eqdef \left\{
    \begin{array}{ll}
      \int_{x \in X}^- f (x) d\mu
      & \text{if $f$ is $\mu$-bounded} \\
      \infty & \text{otherwise.}
    \end{array}
                 \right.
\end{align}

We say that a topological space is \emph{coherent} if and only if the
intersection of any two compact saturated subets is compact (and saturated).
\begin{lemma}
  \label{lemma:int+}
  Let $\mu$ be a non-zero $\tau$-smooth measure on a topological space
  $X$.  The upper integral (\ref{eq:int+}) is:
  \begin{enumerate}
  \item additive if $X$ is coherent: for all upper semicontinuous maps
    $f$, $g$ from $X$ to $\creal$,
    $\int_{x \in X}^+ (f (x) + g (x)) d\mu = \int_{x \in X}^+ f (x)
    d\mu + \int_{x \in X}^+ g (x) d\mu$;
  \item $\cdot_r$-homogeneous: for every upper semicontinuous map
    $f \colon X \to \creal$, for every $a \in \creal$,
    $a \cdot_r \int_{x \in X}^+ f (x) d\mu = \int_{x \in X}^+ (a
    \cdot_r f (x)) d\mu$;
  \item Scott-cocontinuous if $\mu$ is also bounded: for every
    filtered family ${(f_i)}_{i \in I}$ of upper semicontinuous maps
    from $X$ to $\creal$,
    $\int_{x \in X}^+ \finf_{i \in I} f_i (x) d\mu = \finf_{i \in I}
    \int_{x \in X}^+ f_i (x) d\mu$;
  \item above the lower integral: for every measurable map
    $g \colon X \to \creal$, for every upper semicontinuous map
    $f \colon X \to \creal$ such that $g \leq f$ on
    $E \cap \supp \mu$, where $E$ is any measurable support of $\mu$,
    $\int_{x \in X}^- g (x) d\mu \leq \int_{x \in X}^+ f (x) d\mu$;
  \item for every $\mu$-bounded upper semicontinuous map
    $g \colon X \to \creal$, witnessed by $Q$,
    $\int_{x \in X}^- g (x) d\mu = \int_{x \in X}^+ g (x) d\mu$ is the
    usual Lebesgue integral
    $\int_{x \in X} g (x) \One {Q \cap \supp \mu} (x) d\mu$;
  \end{enumerate}
\end{lemma}
\begin{proof}
  We prove item~5 first.  When $g$ is $\mu$-bounded, witnessed by $Q$,
  we can define a new measurable map $g . \One {Q \cap \supp \mu}$,
  which maps every $x \in Q \cap \supp \mu$ to $g (x)$, and all other
  points to $0$.  Then $g . \One {Q \cap \supp \mu}$ is bounded, and
  coincides with $g$ on $Q \cap \supp \mu$.  Since the latter is a
  support of $\mu$, it is an easy exercise, using
  (\ref{eq:int-:choquet}), to show that $\int_{x \in X}^- g (x) d\mu$
  is equal to
  $\int_{x \in X}^- g (x) . \One {Q \cap \supp \mu} (x) d\mu$, which
  is the ordinary Lebesgue integral
  $\int_{x \in X} g (x) . \One {Q \cap \supp \mu} (x) d\mu$.

  1. If $f$ and $g$ are both $\mu$-bounded, witnessed respectively by
  $Q$ and $Q'$, then so is $f+g$, witnessed by $Q \cap Q'$.  The
  latter is measurable, and compact saturated since $X$ is coherent.
  It is also a support of $\mu$, since $Q$ (or $Q'$) is measurable.
  The claim then follows from Lemma~\ref{lemma:int-}, item~1.

  If, say, $f$ is not $\mu$-bounded, then for every measurable compact
  saturated support $Q$ of $\mu$, there is a point
  $x \in Q \cap \supp \mu$ such that $f (x) = \infty$ by
  Corollary~\ref{corl:muunbound}.  Then, $f (x) + g (x)$ is also equal
  to $\infty$, showing that $f+g$ is not $\mu$-bounded either.  In
  particular, $\int_{x \in X}^+ (f (x) + g (x)) d\mu$ and
  $\int_{x \in X}^+ f (x) d\mu + \int_{x \in X}^+ g (x) d\mu$ are both
  equal to $\infty$.

  2. If $f$ is $\mu$-bounded and $a \neq \infty$, then $a \cdot_r f$
  is also $\mu$-bounded: for every measurable compact saturated
  support $Q$ of $\mu$, $f$ and therefore $a \cdot_r f$ is bounded on
  $Q \cap \supp \mu$.  Then the claim follows from
  Lemma~\ref{lemma:int-}, item~2, and the fact that $\cdot_\ell$ and
  $\cdot_r$ both coincide with the ordinary product on $\Rp$.
  
  If $a = \infty$, then by definition
  $\infty \cdot_r \int_{x \in X}^+ f (x) d\mu = \infty$, since
  $\infty$ is absorbing for $\cdot_r$; and
  $\int_{x \in X}^+ (\infty \cdot_r f (x)) d\mu = \int_{x \in X}^+
  \infty d\mu = \infty$.  The latter equality follows from the fact
  that the constant map $\infty$ is not $\mu$-bounded; indeed, for
  every measurable compact saturated support $Q$ of $\mu$,
  $Q \cap \supp \mu$ is non-empty by Lemma~\ref{lemma:Qsupp}, so that
  $\infty$ is not bounded on that set.

  If $f$ is not $\mu$-bounded but $a \in \Rp$, then for every
  measurable compact saturated support $Q$ of $\mu$, there is a point
  $x \in Q \cap \supp \mu$ such that $f (x) = \infty$ by
  Corollary~\ref{corl:muunbound}.  Then $a \cdot_r f (x) = \infty$ as
  well.  This shows that $a \cdot_r f$ is not $\mu$-bounded either.
  It follows that $\int_{x \in X}^+ (a \cdot_r f (x)) d\mu = \infty$,
  while
  $a \cdot_r \int_{x \in X}^+ f (x) d\mu = a \cdot_r \infty = \infty$.

  3. First, the pointwise infimum $f \eqdef \finf_{i \in I} f_i$ of
  upper semicontinuous maps $f_i$'s is upper semicontinuous.  Let us write $i
  \preceq j$ if and only if $f_i \leq f_j$.

  If $f_{i_0}$ is $\mu$-bounded for some $i_0 \in I$, then
  $f_i \leq f_{i_0}$ is also $\mu$-bounded for every $i \preceq i_0$,
  and witnessed by the same measurable compact saturated set $Q$.
  Similarly, $f$ is also $\mu$-bounded, witnessed by $Q$.  We let $r$
  be an upper bound of $f_{i_0}$ on $Q \cap \supp \mu$.  Then, using
  item~5,
  \begin{align*}
    \int_{x \in X}^+ f (x) d\mu
    & = \int_{x \in X} f (x) \One {Q \cap \supp \mu} (x) d\mu \\
    & = r \mu (X) - \int_{x \in X} (r - f (x)) \One {Q \cap \supp \mu} (x) d\mu.
  \end{align*}
  Indeed, the map $(r - f (\_)) \One {Q \cap \supp \mu}$ also takes
  its values in $\Rp$, and the sum of
  $\int_{x \in X} f (x) \One {Q \cap \supp \mu} (x) d\mu$ and of
  $\int_{x \in X} (r - f (x)) \One {Q \cap \supp \mu} (x) d\mu$ is
  equal to
  $\int_{x \in X} r \One {Q \cap \supp \mu} (x) d\mu = r \mu (Q \cap
  \supp \mu) = r \mu (X)$, since $Q \cap \supp \mu$ is a measurable
  support of $\mu$.  Since integration of lower semicontinuous maps
  with respect to a $\tau$-smooth measure is Scott-continuous, as in
  Lemma~\ref{lemma:int-}, item~4, we obtain:
  \begin{align*}
    \int_{x \in X}^+ f (x) d\mu
    & = r \mu (X) - \dsup_{i \preceq i_0} \int_{x \in X} (r - f_i (x))
      \One {Q \cap \supp \mu} (x)  d\mu \\
    & = \finf_{i \preceq i_0} \int_{x \in X} f_i (x) \One {Q \cap \supp \mu}
      (x) d\mu \\
    & = \finf_{i \preceq i_0} \int_{x \in X}^+ f_i (x) d\mu 
    = \finf_{i \in I} \int_{x \in X}^+ f_i (x) d\mu.
  \end{align*}
  If no $f_i$ is $\mu$-bounded, then $f$ cannot be $\mu$-bounded
  either, as we now claim.  If $f$ is $\mu$-bounded, witnessed by $Q$,
  then, let $r \in \Rp$ be such that for every
  $x \in Q \cap \supp \mu$, $f (x) < r$.  Since
  $f = \finf_{i \in I} f_i$, every point $x$ of $Q \cap \supp \mu$ is
  in the open set $f_i^{-1} ([0, r[)$ for some $i \in I$.  The family
  ${(f_i^{-1} ([0, r[))}_{i \in I}$ is then an open cover of
  $Q \cap \supp \mu$.  The intersection of a compact set and of a
  closed set is compact, so $Q \cap \supp \mu$ is compact, and therefore
  ${(f_i^{-1} ([0, r[))}_{i \in I}$ has a finite subcover.  Since
  ${(f_i^{-1} ([0, r[))}_{i \in I}$ is a directed family, we can
  assume that this subcover consists of just one open set
  $f_i^{-1} ([0, r[)$.  But that implies that $f_i$ is bounded on $Q
  \cap \supp \mu$, hence $\mu$-bounded,
  a contradiction.

  Hence we have proved that $f$ is not $\mu$-bounded, so
  $\int_{x \in X}^+ f (x) d\mu = \infty$, which is then vacuously
  equal to $\finf_{i \in I} \int_{x \in X}^+ f_i (x) d\mu$.
  
  4. When $f$ is $\mu$-bounded, witnessed by $Q$,
  $\int_{x \in X}^+ f (x) d\mu$ is equal to
  $\int_{x \in X}^- f (x) d\mu$, hence to the ordinary integral
  $\int_{x \in X} f (x) . \One {Q \cap \supp \mu} (x) d\mu$ by item~5.
  Since $E$ is a measurable support of $\mu$,
  $E \cap Q \cap \supp \mu$ is also a (measurable) support of $\mu$,
  so the latter is also equal to
  $\int_{x \in X} f (x) . \One {E \cap Q \cap \supp \mu} (x) d\mu$.
  Since $g$ is below $f$ on $E \cap \supp \mu$,
  $g \cdot \One {E \cap Q \cap \supp \mu}$ is (bounded and) below
  $f \cdot \One {E \cap Q \cap \supp \mu}$, so
  $\int_{x \in X}^+ f (x) d\mu = \int_{x \in X} f (x) . \One {E \cap Q
    \cap \supp \mu} (x) d\mu$ is larger than or equal to
  $\int_{x \in X} g (x) . \One {E \cap Q \cap \supp \mu} (x) d \mu$,
  and the latter is equal to $\int_{x \in X}^- g (x) d\mu$ by a
  similar argument.
  
  If $f$ is not $\mu$-bounded, then
  $\int_{x \in X}^+ f (x) d\mu = \infty$, and the claim is trivial.
\end{proof}

We let $R \eqdef \Icreal$, and we fix a topological space $X$.  For
every $h \in \Lform^R X$, for every $x \in X$, $h (x)$ is an interval
$[h^- (x), h^+ (x)]$.  The function $h^-$ is lower semicontinuous.
Indeed, for every $r \in \Rp \diff \{0\}$,
${(h^-)}^{-1} (]r, \infty]) = h^{-1} (\uuarrow [r, \infty])$: for
every $[a, b] \in \Icreal$, $[r, \infty] \ll [a, b]$ if and only if
$r < a$.  Symmetrically, the function
$h^+$ is upper semicontinuous.
Our preparatory steps on the lower and upper integrals then allow us
to make sense of the following definition.  The fact that
$\int_{x \in X}^- h^- (x) d\mu \leq \int_{x \in X}^+ h^+ (x) d\mu$ is
by Lemma~\ref{lemma:int+}, item~4.

\begin{definition}
  \label{defn:meas:rval}
  For every $\tau$-smooth measure $\mu$ on a topological space $X$, we
  define $\rval\mu \colon \Lform^{\Icreal} X \to \Icreal$ by
  $\rval\mu (h) \eqdef [\int_{x \in X}^- h^- (x) d\mu, \int_{x \in
    X}^+ h^+ (x) d\mu]$.
\end{definition}


\begin{proposition}
  \label{prop:rval}
  For every non-zero, bounded $\tau$-smooth measure $\mu$ on a
  coherent topological space $X$, $\rval\mu$ is a continuous
  $\Icreal$-valuation.
\end{proposition}
\begin{proof}
  First, $\rval\mu$ is linear by Lemma~\ref{lemma:int-} (items~1
  and~2) and Lemma~\ref{lemma:int+} (items~1 and~2).  We verify that
  It it Scott-continuous.  Let ${(h_i)}_{i \in I}$ be a directed
  family in $\Lform^{\Icreal} X$, with supremum $h$.  We aim to show
  that $\rval\mu (h) = \dsup_{i \in I} \rval\mu (h_i)$.  On the one
  hand, $h^- = \dsup_{i \in I} h_i^-$, so
  $\int_{x \in X}^- h^- (x) d\mu = \dsup_{i \in I} \int_{x \in X}^-
  h_i^- (x) d\mu$ by Lemma~\ref{lemma:int-}, item~4.  On the other
  hand, $h^+ = \finf_{i \in I} h_i^+$, so
  $\int_{x \in X}^+ h^+ (x) d\mu = \finf_{i \in I} \int_{x \in X}^+
  h_i^+ (x) d\mu$ by Lemma~\ref{lemma:int+}, item~3.
\end{proof}

\begin{remark}
  \label{rem:rval}
  We think of $\rval\mu$ as being really the measure $\mu$, seen as a
  continuous $\Icreal$-valuation.  Note in particular that for every
  bounded continuous map $h \colon X \to \Rp$,
  $\rval\mu ([h, h]) = \left[\int_{x \in X} h (x) d\mu, \int_{x \in X}
    h (x) d\mu\right]$.
\end{remark}

\section{Continuous $R$-valuations and measures III: continuous $\Icreal$-valuations as
  approximations of measures}
\label{sec:repr-meas-iii}

Let us say that $[a, b]$ \emph{approximates} $x$ if and only if
$a \leq x \leq b$, and that a continuous map
$h \in \Lform^{\Icreal} X$ \emph{approximates} a measurable map
$f \colon X \to \creal$ if and only if $h (x)$ approximates $f (x)$ for
every $x \in X$.

We will say that a continuous $\Icreal$-valuation $\nu$
\emph{approximates} a measure $\mu$ on $X$ if and only if, for every
measurable map $f \colon X \to \creal$ and for every
$h \in \Lform^{\Icreal} X$ that approximates $f$, $\nu (h)$
approximates $\int_{x \in X}^- f (x) d\mu$.

\begin{lemma}
  \label{lemma:approx:correct}
  For every non-zero, bounded $\tau$-smooth measure $\mu$ on a
  coherent topological space $X$, $\rval\mu$ approximates $\mu$.
\end{lemma}
\begin{proof}
  We consider any measurable map $f \colon X \to \creal$ and any
  $h \in \Lform^{\Icreal} X$ that approximates $f$.  We write $h (x)$
  as $[h^- (x), h^+ (x)]$ for every $x \in X$, so that
  $h^- \leq f \leq h^+$.  Then
  $\int_{x \in X}^- h^- (x) d\mu \leq \int_{x \in X}^- f (x) d\mu \leq
  \int_{x \in X}^+ h^+ (x) d\mu$, where the last inequality is by
  Lemma~\ref{lemma:int+}, item~4 applied to $f \leq h^+$.
\end{proof}

Our objective is now to show that $\rval\mu$ is the most precise,
namely the largest, continuous $\Icreal$-valuation that approximates
$\mu$, under some reasonable assumptions.  This will notably hold when
the ambient space $X$ is compact Hausdorff and second-countable, for
example $[0, 1]$ with its usual, metric topology. 
More generally, this will hold when $X$ is stably
compact, second-countable, and contains a sufficiently nice support
$K$ of $\mu$.

The value of restricting to second-countable spaces is the following.
\begin{lemma}
  \label{lemma:2ndcount}
  Let $X$ be a topological space, and $\mathcal B$ be a base of its
  topology that is closed under finite unions.
  \begin{enumerate}
  \item For every compact saturated subset $Q$ of $X$, and every open
    neighborhood $U$ of $Q$, there is a $V \in \mathcal B$ such that
    $Q \subseteq V \subseteq U$.
  \item Every compact saturated subset of $X$ is equal to the
    intersection of the sets in $\mathcal B$ that contain it.
  \end{enumerate}
  In particular, if $X$ is second-countable, then every compact
  saturated subset of $X$ is measurable.
\end{lemma}
\begin{proof}
  (i) We write $U$ as the union of the sets $V \in \mathcal B$ that are
  included in $V$.  This forms an open cover of $Q$, from which we can
  extract a finite subcover.  Since $\mathcal B$ is closed under
  finite unions, there is a $V \in \mathcal B$ that contains $Q$ and
  is included in $U$.

  (ii)  Let $Q$ be compact saturated in $X$.  Since $Q$ is saturated,
  $Q$ is the intersection of its open neighborhoods $U$.  Then claim~2
  follows from 1.

  When $\mathcal B$ is countable, $Q$ is then a countable intersection
  of open sets, so $Q$ is measurable.
\end{proof}

A space $X$ is \emph{stably compact} if and only if it is sober,
locally compact, compact, and coherent.  We let $X^\patch$ denote $X$
with its \emph{patch topology}, which is the smallest topology that
contains the original open subsets of $X$ and the complements of
compact saturated subsets of $X$.  When $X$ is stably compact, and
$\leq$ is its specialization ordering, $(X^\patch, \leq)$ is a
\emph{compact pospace}, meaning that $X^\patch$ is compact Hausdorff,
and that the graph of $\leq$ is closed in $X^\patch \times X^\patch$.
We say that a subset of $X$ is \emph{patch-open} if it is open in
$X^\patch$.  Similarly, we use the terms \emph{patch-closed},
\emph{patch-compact}.  If $X$ is stably compact, then patch-closed and
patch-compact are synonymous.  We should add that the original open
subsets of $X$ can be recovered as those patch-open subsets that are
upwards-closed with respect to $\leq$.

\begin{example}
  \label{exa:IR:scomp}
  $\Icreal$ is stably compact in its Scott topology.
  Indeed, it is a continuous dcpo in which any pair of elements
  $[a, b]$ and $[c, d]$ with an upper bound (namely, such that
  $[a, b] \cap [c, d] \neq \emptyset$, or equivalently
  $\max (a, c) \leq \min (b, d)$) has a least upper bound (which is
  $[\max (a, c), \min (b, d)]$).  That kind of continuous dcpo is
  called a \emph{bc-domain}, and every bc-domain is stably compact
  \cite[Fact~9.1.6]{goubault13a}.

  Reasoning similarly, the larger dcpo $\IRbb$ of all closed intervals
  in $\real \cup \{-\infty, \infty\}$, ordered by reverse inclusion,
  is also a bc-domain, hence is also stably compact.  (To make it
  clear, note that $\IRbb$ not only contains the usual intervals
  $[a, b]$ with $a, b \in \real$, $a \leq b$, but also $[-\infty, b]$,
  $[a, \infty]$ with $a, b \in \real$; finally, it has a least element
  $[-\infty, \infty]$.)

  Both bc-domains are second-countable as well.  Indeed, as continuous
  dcpos, they have a basis $\mathcal B$ of intervals with rational
  endpoints, and then the set of Scott-open sets $\uuarrow b$, $b \in
  \mathcal B$, forms a countable base of the Scott topology.
\end{example}

Patch-compact subsets $K$ of stably compact subspaces $X$ enjoy many
nice properties.  For example, their downward closure $\dc K$ in $X$
is closed \cite[Exercise~9.1.43]{goubault13a}.  In fact, we have the
following, where $K$ is \emph{order-convex} if and only if for all $x,
y, z$ such that $y \leq x \leq z$, if $y, z \in K$ then $x \in K$.
\begin{lemma}
  \label{lemma:dcK}
  For every compact, order-convex subset $K$ of a stably compact space
  $X$, $K$ is patch-compact if and only if $\dc K$ is closed.  In that
  case, $\dc K$ is the closure $cl (K)$ of $K$ in $X$, and $K = \upc K
  \cap \dc K$.
\end{lemma}
\begin{proof}
  We only have to show that if $K$ is compact and order-convex and if
  $\dc K$ is closed, then it is patch-compact.  This will be a
  consequence of the last equality $K = \upc K \cap \dc K$, since
  $\upc K$ is compact saturated, and we have assumed that $\dc K$ is
  closed, hence both are patch-closed; then $K$ is patch-closed, too,
  hence patch-compact.

  The inclusion $K \subseteq \upc K \cap \dc K$ is clear.  Conversely,
  every $x \in \upc K \cap \dc K$ is such that $y \leq x \leq z$ for
  some $y, z \in K$, so order-convexity implies $y=z$, and therefore
  also $x=y=z$.  In particular, $x$ is in $K$.
\end{proof}

We will say that a subset $K$ of a space $X$ is Hausdorff if and only
if it is Hausdorff as a subspace, namely with the subspace topology
inherited from $X$.  Since the specialization ordering of a Hausdorff
space is equality, every Hausdorff subset is trivially order-convex.

\begin{example}
  \label{exa:I:IR:embed}
  Let $X \eqdef \IRbb$.  Then the unit interval $K \eqdef [0, 1]$ embeds
  into $X$, provided that we equate every point $x \in [0, 1]$ with
  the interval $[x, x]$ in $\IRbb$.  It is Hausdorff, hence
  order-convex.  Its downward closure $\dc K$ in $X$ is the set of all
  intervals $[a, b]$ such that $[a, b] \cap [0, 1] \neq \emptyset$, or
  equivalently such that $\max (a, 0) \leq \min (b, 1)$, or
  equivalently $a \leq 1$ and $b \geq 0$.  Then $\dc K$ is closed: for
  every directed family ${([a_i, b_i])}_{i \in I}$ in $\dc K$, its
  supremum $[a, b]$ is such that $a = \dsup_{i \in I} a_i \leq 1$ and
  $b = \finf_{i \in I} b_i \geq 0$.  By Lemma~\ref{lemma:dcK},
  $[0, 1]$ is patch-compact in $\IRbb$.
\end{example}

When $K$ is patch-compact in a stably compact
space $X$, we have the serendipitous property that $K$, with the
subspace topology, is stably compact, and that the patch topology on
$K$ is the subspace topology inherited from $X^\patch$
\cite[Proposition~9.3.4]{goubault13a}; also, the specialization
ordering on $K$ is the restriction of that on $X$.

The previous remark, together with the fact that $K^\patch = K$ if $K$
is compact and Hausdorff (since every compact set is already closed in
$K$), entails the following.
\begin{lemma}
  \label{lemma:Kpatch}
  Let $X$ be a stably compact space, and $K$ be a Hausdorff,
  patch-compact subset of $X$.  Then $K$ has both the subspace
  topology inherited from $X$, and the subspace topology inherited
  from $X^\patch$.  \qed
\end{lemma}
We also note that every Hausdorff subset is order-convex, since the
specialization ordering on any Hausdorff space is equality.

\begin{proposition}
  \label{prop:approx}
  Let $K$ be a Hausdorff, patch-compact subset of a stably compact,
  second-countable space $X$.
  Let $\mu$ be a non-zero measure on $X$ supported on $K$, and $\nu$
  be a continuous $\Icreal$-valuation on $X$.  If $\nu$ approximates
  $\mu$, then $\nu \leq \rval\mu$.
\end{proposition}
\begin{proof}
  We first note that, since $X$ is second-countable, every measure on
  $X$ is $\tau$-smooth, in particular $\mu$.

  Let $h$ be an arbitrary continuous map in $\Lform^{\Icreal} X$, let
  us write $h (x)$ as $[h^- (x), h^+ (x)]$ for every $x \in X$ (for
  short, $h = [h^-, h^+]$), and $\nu (h)$ as $[\nu^- (h), \nu^+ (h)]$.
  We must show that $\nu^- (h) \leq \int_{x \in X}^- h^- (x) d\mu$,
  and that $\int_{x \in X}^+ h^+ (x) d\mu \leq \nu^+ (h)$.

  For the first claim, we note that $h (x)$ is the supremum of the
  chain of maps $[\min (h^-, r), h^+]$, $r \in \Rp$.  For each
  $r \in \Rp$, $[\min (h^-, r), h^+]$ approximates the bounded lower
  semicontinuous (hence measurable) map $\min (h^-, r)$.  Since it is
  bounded,
  $\int_{x \in X}^- \min (h^- (x), r) d\mu = \int_{x \in X} \min (h^-
  (x), r) d\mu$.  By assumption, $\nu ([\min (h^-, r), h^+])$
  approximates $\int_{x \in X} \min (h^- (x), r) d\mu$.  In
  particular,
  $\nu^- ([\min (h^-, r), h^+]) \leq \int_{x \in X} \min (h^- (x), r)
  d\mu$.  By taking suprema as $r$ grows to infinity, and using the
  Scott-continuity of $\nu$, hence of $\nu^-$,
  $\nu^- (h) \leq \int_{x \in X}^- h^- (x) d\mu$.

  For the second claim, we distinguish two cases.  If $h^+$ is
  $\mu$-bounded, then since $h$ approximates $h^+$, $\nu (h)$
  approximates $\int_{x \in X}^- h^+ (x) d\mu$, which is equal to
  $\int_{x \in X}^+ h^+ (x) d\mu$ by Lemma~\ref{lemma:int+}, item~5.
  In particular, $\int_{x \in X}^+ h^+ (x) d\mu \leq \nu^+ (h)$.

  The only case where we have to work a bit is the final case, when
  $h^+$ is $\mu$-unbounded.  We let $Q \eqdef \upc K$.  This is a
  compact saturated subset of $X$, and since $X$ is second-countable,
  $Q$ is measurable by Lemma~\ref{lemma:2ndcount}.  Moreover, $Q$ is a
  support of $\mu$, since $Q$ contains $K$, which is already a support
  of $\mu$.

  $Q \cap \supp \mu$ is then a measurable, compact support of $\mu$.
  We claim that $\supp \mu$ is included in the closure $cl (K)$ of $K$
  in $X$.  Equivalently, we claim that every open set $U$ that
  intersects $\supp \mu$ also intersects $K$.  Since $U$ intersects
  $\supp \mu$, by definition of $\supp \mu$, we have $\mu (U) > 0$.
  If it did not intersect $K$, then $U$ and the empty set would have
  the same intersection with $K$, and that would imply $\mu (U) = \mu
  (\emptyset) = 0$, which is impossible.

  Since $\supp \mu \subseteq cl (K)$, we obtain that
  $Q \cap \supp \mu$ is included in $\upc K \cap cl (K)$.  But
  $cl (K) = \dc K$ and $K = \upc K \cap \dc K$ by
  Lemma~\ref{lemma:dcK}.  (Recall that $K$ is order-convex since
  Hausdorff.)  Therefore, $Q \cap \supp \mu$ is included in $K$.  By
  Corollary~\ref{corl:muunbound}, there is a point $x_0$ in
  $Q \cap \supp \mu$, hence in $K$, such that $h^+ (x_0) = \infty$.

  Since $K$ is compact Hausdorff hence locally compact, $x_0$ has a
  base of compact neighborhoods ${(K_i)}_{i \in I}$ in $K$.  Each
  $K_i$ is compact hence closed in $K$ (since $K$ is Hausdorff).  Let
  $C_i$ be the closure of $K_i$ in $X$.  Then $C_i \cap K = K_i$: the
  inclusion $K_i \subseteq C_i \cap K$ is clear; conversely, since
  $K_i$ is closed in $K$, we can write it as $C \cap K$ for some
  closed subset $C$ of $X$, and then $C \supseteq C_i$, so
  $K_i = C \cap K \supseteq C_i \cap K$.

  The family ${(K_i)}_{i \in I}$ is filtered: for all $i, j \in I$,
  $K_i \cap K_j$ is a compact neighborhood of $x_0$, using the fact
  that the intersection of two compact sets in a Hausdorff space is
  compact; hence $K_i \cap K_j$ contains some $K_k$, $k \in I$.  It
  follows that ${(C_i)}_{i \in I}$ is a filtered family of closed
  subsets of $X$.

  Let $C \eqdef \fcap_{i \in I} C_i$.  This is a closed subset of $X$
  containing $x_0$.  We claim that $C$ is exactly the downward closure
  of $x_0$ in $X$.  It only remains to show that
  $C \subseteq \dc x_0$.  Let us assume the contrary: for some
  $x \in \fcap_{i \in I} \dc K_i$, $x \not\leq x_0$.  Then $x_0$ is in
  $(X \diff \upc x) \cap K$, which is open in $K$.  Indeed, $\upc x$
  is compact saturated hence patch-closed in $X$, so $X \diff \upc x$
  is open in $X^\patch$, and therefore $(X \diff \upc x) \cap K$ is
  open in $K$, using Lemma~\ref{lemma:Kpatch}.  Since
  ${(K_i)}_{i \in I}$ is a base of neighborhoods of $x_0$ in $K$, some
  $K_i$ is included in $(X \diff \upc x) \cap K$.  This is impossible,
  since $x \in \dc K_i$.


  For every $i \in I$, let $h_i$ be the function that maps every
  $x \in C_i$ to $\infty$, and all other points to $h^+ (x)$.  This is
  an upper semicontinuous map, since
  $h_i^{-1} ([0, r[) = {h^+}^{-1} ([0, r[) \diff C_i$ for every
  $r \in \Rp$.  The family ${(h_i)}_{i \in I}$ is filtered, since
  ${(C_i)}_{i \in I}$ is a filtered family of sets.  Moreover, for
  every $x \in X$, $\finf_{i \in I} h_i (x) = h^+ (x)$.  If
  $x \leq x_0$, we argue as follows.  First,
  $x \in \dc x_0 = C = \fcap_{i \in I} C_i$, so that
  $\finf_{i \in I} h_i (x) = \finf_{i \in I} \infty = \infty$, while
  $h^+ (x) \geq h^+ (x_0) = \infty$, since upper semicontinuous maps
  are antitonic.  If $x \not\leq x_0$, then $x$ is not in
  $C = \fcap_{i \in I} C_i$, so $x$ is not in $C_i$ for some
  $i \in I$, and therefore $h_i (x) = h^+ (x)$.  This implies that
  $\finf_{i \in I} h_i (x) \leq h^+ (x)$, while the reverse inequality
  is obvious.
  
  For every $i \in I$, $[h^-, h_i]$ approximates $h_i$, so
  $\nu ([h^-, h_i])$ approximates $\int_{x \in X}^- h_i (x) d\mu$.  In
  particular, $\int_{x \in X}^- h_i (x) d\mu \leq \nu^+ ([h^-, h_i])$.
  However, we claim that the left-hand side is equal to $\infty$, so
  that $\nu^+ ([h^-, h_i]) = \infty$.  Indeed, $h_i (x)$ is equal to
  $\infty$ on $C_i$, hence on $K_i \subseteq C_i$, hence on the even
  smaller set $U_i \cap K$, so
  $\int_{x \in X}^- h_i (x) d\mu \geq \infty.\mu (U_i \cap K)$.  (The
  latter makes sense because $K = \upc K \cap \dc K$ by
  Lemma~\ref{lemma:dcK}, $\upc K$ is compact saturated hence
  measurable by Lemma~\ref{lemma:2ndcount}, $U_i$ is open and $\dc K$
  are closed, hence are measurable.)  Since $K$ is a support of $\mu$,
  $\mu (U_i \cap K) = \mu (U_i)$.  We have an open set $U_i$ that
  intersects $\supp \mu$ (at $x_0$): by definition of the support,
  $\mu (U_i) > 0$.  (Namely, if we had $\mu (U_i)=0$, then $U_i$ would
  be included in the largest open subset with zero $\mu$-measure,
  which is the complement of $\supp \mu$ by definition.)  It follows
  that
  $\int_{x \in X}^- h_i (x) d\mu \geq \infty.\mu (K_i) \geq \infty
  . \mu (U_i) = \infty$.

  Hence we have shown that $\nu^+ ([h^-, h_i]) = \infty$ for every
  $i \in I$.  Taking suprema, and recalling that
  $h^+ = \finf_{i \in I} h_i$, hence that
  $\dsup_{i \in I} [h^-, h_i] = [h^-, h^+] = h$, we obtain that
  $\nu^+ (h) = \infty$.  The inequality
  $\int_{x \in X}^+ h^+ (x) d\mu \leq \nu^+ (h)$ then follows
  trivially.
\end{proof}

\begin{theorem}
  \label{thm:rval}
  Let $K$ be a Hausdorff, patch-compact subset of a stably compact,
  second-countable space $X$, and $\mu$ be a non-zero, bounded measure
  supported on $K$.  Then $\mu$ is $\tau$-smooth and $\rval\mu$ is the
  largest (``most precise'') continuous $\Icreal$-valuation that
  approximates $\mu$.
\end{theorem}
\begin{proof}
  First, $\mu$ is $\tau$-smooth because $X$ is second-countable.
  Second, $\rval\mu$ is a continuous $\Icreal$-valuation by
  Proposition~\ref{prop:rval}, it approximates $\mu$ by
  Lemma~\ref{lemma:approx:correct}, and it is largest by
  Proposition~\ref{prop:approx}.
\end{proof}


\section{The Lebesgue $R$-valuation on the unit interval}
\label{sec:lebesgue-r-valuation}

Let $\lambda$ be Lebesgue measure on $[0, 1]$.  By
Theorem~\ref{thm:rval} with $X \eqdef K \eqdef [0, 1]$, $\rval\lambda$
is the most precise continuous $\Icreal$-valuation that approximates
$\lambda$.  However, $\rval\lambda$ is not in
$\Val^{\Icreal}_m ([0, 1])$, by the following argument, whose details
we leave to the reader.  If $\rval\lambda$ were minimal, then its view
from the left would be in $\Val^{\creal}_m ([0, 1])$, so $\lambda$
would be a minimal valuation.  Any minimal valuation is
point-continuous, in the sense of Heckmann \cite{heckmann95}, because
every simple valuation is point-continuous, and point-continuous
valuations are closed under directed suprema.  However, a valuation
$\nu$ is point-continuous if and only if for every open set $U$, for
every real number $r$ such that $0 \leq r < \nu (U)$, there is a
finite subset $A$ of $U$ such that $\nu (V)$ for every open
neighborhood $V$ of $A$; and $\lambda$ fails to have this property,
since every finite subset has open neighborhoods of arbitrarily small
$\lambda$-measure.

Instead, we consider the image measure $j [\lambda]$ on
$\IRbb$, where $j$ is the usual embedding of $[0, 1]$ inside $\IRbb$,
mapping $x$ to the interval $[x, x]$.  We will show that, contrarily
to $\lambda$, $j [\lambda]$ \emph{is} minimal.

This may sound somewhat paradoxical, considering that both have the
same effect: drawing an interval at random with respect to measure
$j [\lambda]$ means drawing an interval of the form $[x, x]$ with
$x \in [0, 1]$ with probability $1$, where $x$ is drawn uniformly at
random in $[0, 1]$, hence works just like $\lambda$; only the ambient
space differs ($\IRbb$ instead of $[0, 1]$).

We will say that $j [\lambda]$ is the \emph{Lebesgue valuation on the
  unit interval} in $\IRbb$.  By Theorem~\ref{thm:rval} with
$X \eqdef \IRbb$ and $K \eqdef [0, 1]$, $\rval {j [\lambda]}$ is the
continuous $\Icreal$-valuation that approximates $j [\lambda]$ in the
most precise possible way.  It is a bounded, non-zero, and
$\tau$-smooth measure (because $\IRbb$ is second-countable, see
Example~\ref{exa:IR:scomp}).  The objective of this section is to show
that $\rval {j [\lambda]}$ \emph{is} in $\Val^{\Icreal}_m (\IRbb)$.

To this end, we will show the stronger statement that
$\rval {j [\lambda]}$ is the directed supremum of a countable chain of
simple $\Icreal$-valuations $\lebesgue_n$, $n \in \nat$.

We define $\lebesgue_n$ on $\IRbb$ as
$\sum_{i=1}^{2^n} [\frac 1 {2^n}, \frac 1 {2^n}] \times \delta_{[\frac
  {i-1} {2^n}, \frac i {2^n}]}$---which we will write more simply as
$\sum_{i=1}^{2^n} \frac 1 {2^n} \delta_{[\frac {i-1} {2^n}, \frac i
  {2^n}]}$, equating points $a \in \real \cup \{-\infty, \infty\}$
with intervals $[a, a]$.

\begin{lemma}
  \label{lemma:lambdan}
  The simple $\Icreal$-valuations $\lebesgue_n$, $n \in \nat$, form an
  ascending chain in $\Val^{\Icreal}_\fin (\IRbb)$.
\end{lemma}
\begin{proof}
  It suffices to show that $\lebesgue_n \leq \lebesgue_{n+1}$.  For every
  $h \in \Lform^{\Icreal} (\IRbb)$,
  \begin{align*}
    \lebesgue_n (h)
    & = \sum_{i=1}^{2^n} \frac 1 {2^n} \times h \left(\left[\frac {i-1}
      {2^n}, \frac i {2^n}\right]\right) \\
    & = \sum_{i=1}^{2^n} \left(
      \frac 1 {2^{n+1}} \times h \left(\left[\frac {i-1} {2^n}, \frac i {2^n}\right]\right)
      + \frac 1 {2^{n+1}} \times h \left(\left[\frac {i-1} {2^n}, \frac i {2^n}\right]\right)
      \right) \\
    & \leq \sum_{i=1}^{2^n} \left(
      \frac 1 {2^{n+1}} \times h \left(\left[\frac {2i-2} {2^{n+1}}, \frac {2i-1} {2^{n+1}}\right]\right)
      + \frac 1 {2^{n+1}} \times h \left(\left[\frac {2i-1} {2^{n+1}}, \frac
      {2i} {2^{n+1}}\right]\right)
      \right) \\
    & = \lebesgue_{n+1} (h).
  \end{align*}
  The second line is justified by the fact that $\times$ distributes
  over $+$.  The inequality on the third line follows from the fact
  that $[\frac {i-1} {2^n}, \frac i {2^n}]$ is below (contains) both
  $[\frac {2i-2} {2^{n+1}}, \frac {2i-1} {2^{n+1}}]$ and
  $[\frac {2i-1} {2^{n+1}}, \frac {2i} {2^{n+1}}]$, and that $h$,
  product and addition are monotonic.  The last line follows by
  rearranging the sum.
\end{proof}

The chain ${(\lebesgue_n)}_{n \in \nat}$ then has a supremum in $\Val^{\Icreal}
(\IRbb)$, which is in $\Val^{\Icreal}_m (\IRbb)$ by definition of the latter,
since every $\lebesgue_n$ is simple.
\begin{definition}
  \label{defn:lambda}
  Let $\lebesgue$ be $\dsup_{n \in \nat} \lebesgue_n$.
\end{definition}

For every $h \in \Lform^{\Icreal} (\IRbb)$, we have:
\begin{align}
  \label{eq:lebesgue}
  \nonumber
  \lebesgue (h)
  & = \dsup_{n \in \nat} \lebesgue_n (h) \\
  & = \left[
    \dsup_{n \in \nat}
    \sum_{i=1}^{2^n}\frac 1 {2^n} h^- \left(\left[\frac {i-1} {2^n},
    \frac i {2^n}\right]\right),
    \finf_{n \in \nat}
    \sum_{i=1}^{2^n}\frac 1 {2^n} h^+ \left(\left[\frac {i-1} {2^n},
    \frac i {2^n}\right]\right)
    \right].
\end{align}


\begin{theorem}
  \label{thm:rval:lebesgue}
  The continuous $\Icreal$-valuation $\lebesgue$ is the largest
  (``most precise'') continuous $\Icreal$-valuation
  $\rval {j [\lambda]}$ that approximates Lebesgue measure on the unit
  interval in $\IRbb$.

  In particular, $\rval {j [\lambda]} = \lebesgue$ is in
  $\Val^{\Icreal}_m (\IRbb)$.
\end{theorem}
\begin{proof}
  The interval $[0, 1]$, with its usual ordering, is a continuous dcpo, and its
  way-below relation $\ll$ is such that $x \ll y$ if and only if $x=0$
  or $x<y$.  For every $n \in \nat$, for every $x \in [0, 1]$, let
  $j_n^- (x)$ be the largest integer multiple of $\frac 1 {2^n}$
  way-below $x$; explicitly, $j_n^- (x) \eqdef \frac i {2^n}$ if
  $x \in \left]\frac i {2^n}, \frac {i+1} {2^n}\right]$,
  $i \in \{1, 2, \cdots, 2^n-1\}$, and $j_n^- (x) \eqdef 0$ if
  $x \in \left[0, \frac 1 {2^n}\right]$.  By the definition of $\ll$,
  $j_n^-$ is Scott-continuous from $[0, 1]$ to
  $\real \cup \{-\infty, \infty\}$ (both taken with their usual
  orderings).  This implies that $j_n$ is lower semicontinuous from
  $[0, 1]$ to $\real \cup \{-\infty, \infty\}$, with their usual,
  Hausdorff topologies.

  Let also $j_n^+ (x) \eqdef 1 - j_n^- (1-x)$.  The function $j_n^+$
  is an upper semicontinuous map, and
  $j_n^- (x) \leq x \leq j_n^+ (x)$ for every $x \in X$.  This implies
  that, if we define $j_n (x)$ as $[j_n^- (x), j_n^+ (x)]$, $j_n$ is a
  continuous map from $[0, 1]$ to $\IRbb$, and $j_n \leq j$.

  One checks easily that $j_n^- \leq j_{n+1}^-$, hence
  $j_n^+ \geq j_{n+1}^+$, and therefore $j_n \leq j_{n+1}$.  It
  follows that ${(j_n)}_{n \in \nat}$ is an increasing chain of
  continuous maps.  Moreover $\dsup_{n \in \nat} j_n = j$.  Indeed,
  for every $x \in X$, $\dsup_{n \in \nat} j_n^- (x) = x$, and
  $\finf_{n \in \nat} j_n^+ (x) = 1 - \dsup_{n \in \nat} j_n^- (1-x) =
  1-(1-x) = x$.

  Let now $h \eqdef [h^-, h^+]$ be any element of
  $\Lform^{\Icreal} (\IRbb)$.  We wish to show that
  $\rval {j [\lambda]} (h) = \lebesgue (h)$, namely that
  $\left[\int_{y \in \IRbb}^- h^- (y) d j [\lambda], \int_{y \in
      \IRbb}^+ h^+ (y) dj [\lambda]\right] = \lebesgue (h)$.

  Note that $h^-$ is lower semicontinuous, or equivalently
  Scott-continuous from $\IRbb$ to $\Icreal$.  We have:
  \begin{align*}
    \int_{y \in \IRbb}^- h^- (y) dj[\lambda]
    & = \int_{x \in [0, 1]}^- h^- (j (x)) d\lambda
    & \text{by the change of variable formula (\ref{eq:chgvar})} \\
    & = \int_{x \in [0, 1]}^- \dsup_{n \in \nat} h^- (j_n (x))
      d\lambda
    & \text{since $h^-$ is Scott continuous} \\
    & = \dsup_{n \in \nat} \int_{x \in [0, 1]}^- h^- (j_n (x))
      d\lambda
    & \text{by Lemma~\ref{lemma:int-}, item~3 (or~4)}.
  \end{align*}
  The function $h^- \circ j_n$ is piecewise constant: it takes the
  value $h^- (\left[\frac {i-1} {2^n}, \frac i {2^n}\right])$ on the
  open subinterval $\left]\frac {i-1} {2^n}, \frac i {2^n}\right[$,
  $1 \leq i \leq 2^n$.  Since the values in $[0, 1]$ that are not in one of those
  open subintervals, namely the integer multiples of $\frac 1 {2^n}$,
  form a set of Lebesgue measure $0$, they do not contribute to the
  integral, so:
  \begin{align*}
    \int_{y \in \IRbb}^- h^- (y) dj[\lambda]
    & = \dsup_{n \in \nat} \sum_{i=1}^{2^n}
      \frac 1 {2^n} h^- \left(\left[\frac {i-1} {2^n}, \frac i {2^n}\right]\right),
  \end{align*}
  and we recognize the left endpoint of the final interval of
  (\ref{eq:lebesgue}).

  We now deal with the right endpoint.  For that, we need to
  understand what the supports of $j [\lambda]$ are on $\IRbb$.  Let
  $K$ be the image of $[0, 1]$ by $j$ in $\IRbb$.  We have seen in
  Example~\ref{exa:I:IR:embed} that $\dc K$ is closed in $\IRbb$,
  hence that $K$ is patch-compact in $\IRbb$, using
  Lemma~\ref{lemma:dcK}.
  
  For every open subset $U$ of $\IRbb$, $j [\lambda] (U) > 0$ if and
  only if $\lambda (j^{-1} (U)) > 0$, if and only if $j^{-1} (U)$ is
  non-empty.  Indeed, the Lebesgue measure of any non-empty open set
  is non-zero.  Now $j^{-1} (U)$ is empty if and only if $U$ does not
  intersect $K$, if and only if $U$ does not intersect the closure of
  $K$, which is $\dc K$, since $\dc K$ is closed.  This entails that
  $\supp j [\lambda] = \dc K$.
  
  We note that $K$ is a support of $j [\lambda]$.  The easy argument
  is as follows.  First, $K$ is a set of maximal elements of $\IRbb$,
  so $K = \upc K$; by Lemma~\ref{lemma:2ndcount}, it is measurable.
  In order to show that $K$ is a support of $j [\lambda]$, it then
  suffices to observe that $j [\lambda] (K) = 1$, and this follows
  from the fact that
  $j [\lambda] (K) = \lambda (j^{-1} (K)) = \lambda ([0, 1]) = 1$.

  We claim that every compact saturated support $Q$ of $j [\lambda]$
  must contain $K$.  We argue as follows.  By
  Lemma~\ref{lemma:Kpatch}, and since $Q$ is patch-closed in
  $X^\patch$ (where $X \eqdef \IRbb$), $Q \cap K$ is closed in $K$.
  If $Q$ does not contain $K$, then $Q \cap K$ is a proper subset of
  $K$.  Let $U \eqdef K \diff Q$: this is a non-empty open subset of
  $K$, and therefore $j^{-1} (U)$ is a non-empty subset of $[0, 1]$.
  Hence $\lambda (j^{-1} (U)) > 0$, so $j [\lambda] (U) > 0$.  It
  follows that $j [\lambda] (Q \cap K) = 1 - j [\lambda] (U) < 1$.
  Since $K$ is a support of $j [\lambda]$, we obtain that
  $j [\lambda] (Q) = j [\lambda] (Q \cap K) < 1$, and that contradicts
  that $Q$ is a support of $j [\lambda]$.

  It follows that $\upc K$ ($=K$) is the smallest compact saturated
  support of $j [\lambda]$.  In that case, and with
  $\mu \eqdef j [\lambda]$, the definition of $\mu$-boundedness
  simplifies: $h^+$ is $\mu$-bounded if and only if $h^+$ is bounded
  on $\upc K \cap \supp \mu = \upc K \cap \dc K = K$.

  It is even easier to show that $h^+ \circ j$ is $\lambda$-bounded if
  and only if it is bounded on $[0, 1]$.  Indeed,
  $\supp \lambda = [0, 1]$, and therefore the only compact (hence
  closed) support $Q$ of $\lambda$ is $[0, 1]$, so that there is only
  one possible set $Q \cap \supp \lambda$ to be considered, namely
  $[0, 1]$.

  If $h^+$ is $\mu$-unbounded, then
  $\int_{y \in \IRbb} h^+ (y) dj[\lambda] = \infty$ by definition.  By
  Corollary~\ref{corl:muunbound} with $Q \eqdef \upc K = K$, there is
  a point $[x,x] \in Q \cap \supp \mu = K$ such that
  $h^+ ([x,x]) = \infty$.  For every $n \in \nat$, there is a natural
  number $i$ such that $x \in [\frac {i-1} {2^n}, \frac i {2^n}]$,
  $1\leq i\leq 2^n$.  Then
  $[\frac {i-1} {2^n}, \frac i {2^n}] \leq [x, x]$, and since every
  upper semicontinuous map is antitonic,
  $h^+ ([\frac {i-1} {2^n}, \frac i {2^n}]) \geq h^+ ([x,x]) =
  \infty$.  It follows that
  $\sum_{i=1}^{2^n} \frac 1 {2^n} h^+ ([\frac {i-1} {2^n}, \frac i
  {2^n}]) = \infty$.  Since that holds for every $n \in \nat$,
  $\finf_{n \in \nat} \sum_{i=1}^{2^n} \frac 1 {2^n} h^+ ([\frac {i-1}
  {2^n}, \frac i {2^n}])$, which is the right endpoint of the final
  interval of (\ref{eq:lebesgue}), is equal to $\infty$, hence to
  $\int_{y \in \IRbb} h^+ (y) dj[\lambda]$.

  It remains to deal with the case where $h^+$ is $\mu$-bounded, and
  we have seen that this means that $h^+$ is bounded on $K$.
  Let $r \in \Rp$ be such that for every $y \in K$, $h^+ (y) < r$.
  Hence $K$ is included in the open set ${h^+}^{-1} ([0, r[)$.
  
  For every $n \in \nat$, let
  $Q_n \eqdef \upc \{[\frac {i-1} {2^n}, \frac i {2^n}] \mid i \in
  \{1, 2, \cdots, 2^n\}\}$, a compact saturated set that contains $K$.
  (The upward closure of any finite set is compact saturated.)  It is
  also easy to see that any point of $\fcap_{n \in \nat} Q_n$ is of
  the form $[x, x]$ with $x \in [0, 1]$, so
  $\fcap_{n \in \nat} Q_n = K$.  Since $\IRbb$ is sober, it is
  well-filtered, and therefore $K \subseteq {h^+}^{-1} ([0, r[)$
  implies the existence of an $n \in \nat$ such that
  $Q_n \subseteq {h^+}^{-1} ([0, r[)$.  (Here is an alternate argument
  that avoids well-filteredness.  $Q_n$ is compact saturated hence
  closed in ${\IRbb}^\patch$.  The complements of the sets $Q_n$ then
  form an open cover of the complement of ${h^+}^{-1} ([0, r[)$ in
  $\IRbb$.  That complement is closed hence compact in
  ${\IRbb}^\patch$.  We can then extract a finite subcover, and since
  the sets $Q_n$ form a chain, there is a single $n \in \nat$ such
  that the complement of $Q_n$ contains the complement of
  ${h^+}^{-1} ([0, r[)$.)
  
  Let $n_0$ be the natural number $n$ that we have just found.  Then
  we perform the following computation:
  \begin{align*}
    \int_{y \in \IRbb}^+ h^+ (y) dj [\lambda]
    & = \int_{y \in \IRbb}^- h^+ (y) dj [\lambda]
    & \text{since $h^+$ is $j [\lambda]$-bounded} \\
    & = \int_{x \in [0, 1]}^- h^+ (j (x)) d \lambda
    &  \text{by the change of variable formula (\ref{eq:chgvar})} \\
    & = \int_{x \in [0, 1]}^+ h^+ (j (x)) d \lambda.
  \end{align*}
  The latter is justified by the fact that $h^+ \circ j$ is bounded
  (by $r$) on $[0, 1]$, hence is $\lambda$-bounded, as we have seen
  above.

  Let us proceed.  The first step below is justified by the fact that
  $h^+$ is upper semicontinuous, hence Scott-continuous from $\IRbb$
  to $\creal$ with the opposite ordering; in particular, $h^+$ maps
  directed suprema to filtered infima:
  \begin{align*}
    \int_{x \in [0, 1]}^+ h^+ (j (x)) d \lambda
    & = \int_{x \in [0, 1]}^+ \finf_{n \in \nat} h^+ (j_n (x))
      d\lambda \\
    & = \finf_{n \in \nat} \int_{x \in [0, 1]}^+ h^+ (j_n (x))
      d\lambda
    & \text{by Lemma~\ref{lemma:int+}, item~3} \\
    & = \finf_{n \in \nat, n > n_0} \int_{x \in [0, 1]}^+ h^+ (j_n (x))
      d\lambda,
  \end{align*}
  where we have restricted the infimum to the indices above $n_0$ in
  the last line.  (The infimum of a chain coincides with the infimum
  of any coinitial chain.)
  
  For every $n > n_0$, for every $x \in [0, 1]$, $j_n (x)$ is an
  interval of the form $[\frac {i-1} {2^n}, \frac i {2^n}]$ (if $x$ is
  in the interval $]\frac {i-1} {2^n}, \frac i {2^n}[$, or if $x=0$,
  or if $x=1$), or of the form
  $[\frac {i-1} {2^n}, \frac {i+1} {2^n}]$ (if $x$ is exactly
  $\frac i {2^n}$, $1\leq i \leq 2^n-1$).  Since $n > n_0$, whichever
  the case is, $j_n (x)$ is in $Q_{n_0}$, hence in
  ${h^+}^{-1} ([0, r[)$.  This means that $h^+ \circ j_n$ is a bounded
  function, and therefore that
  $\int_{x \in [0, 1]}^+ h^+ (j_n (x)) d\lambda$ is the ordinary
  Lebesgue integral $\int_{x \in [0, 1]} h^+ (j_n (x)) d\lambda$.
  Since $h^+ \circ j_n$ is piecewise constant (as with
  $h^- \circ j_n$, earlier on), that Lebesgue integral is equal to
  $\sum_{i=1}^{2^n} \frac 1 {2^n} h^+([\frac {i-1} {2^n}, \frac i
  {2^n}])$.  Therefore:
  \begin{align*}
    \int_{y \in \IRbb}^+ h^+ (y) dj [\lambda]
    & = \finf_{n \in \nat, n > n_0} \sum_{i=1}^{2^n} \frac 1 {2^n}
      h^+\left(\left[\frac {i-1} {2^n}, \frac i {2^n}\right]\right) \\
    & = \finf_{n \in \nat} \sum_{i=1}^{2^n} \frac 1 {2^n}
      h^+\left(\left[\frac {i-1} {2^n}, \frac i {2^n}\right]\right),
  \end{align*}
  and we recognize the right endpoint of the final interval of
  (\ref{eq:lebesgue}).
\end{proof}






\section{Conclusion}
\label{sec:conc}

We have proposed an extension of the notion of continuous valuation,
or measure, with values in suitable domains beyond $\creal$.  We have
argued that continuous $R$-valuations, where $R$ is a so-called
Abelian d-rag, provide such an extension.  Beyond $\creal$, a
particularly interesting Abelian d-rag is the domain of intervals
$\Icreal$, and we have shown that there is an ample supply of
continuous $\Icreal$-valuations stemming from measures.

There are many pending questions.  For example, is $\Val^R (X)$ a
\emph{continuous} dcpo, provided that $X$ is a continuous dcpo and $R$
is a continuous Abelian d-rag?  Is there a form of the Fubini theorem
for continuous $R$-valuations, beyond the one we have obtained for
minimal $R$-valuations?  None of the usual proof arguments, in realms
of measures or of continuous valuations, seems to apply.

\bibliographystyle{entics}
\bibliography{myrefs}

\end{document}